\documentclass[10pt]{article}
\usepackage{amsmath,amssymb,amsthm}
\usepackage{url}
\usepackage{graphicx}
\usepackage[usenames,dvipsnames]{xcolor}
\usepackage[ocgcolorlinks,colorlinks=true,urlcolor=blue,linkcolor=blue]{hyperref}
\usepackage{enumerate}
\usepackage{verbatim}
\usepackage{booktabs}
\usepackage[nodisplayskipstretch]{setspace}
\usepackage{multirow}
\usepackage{paralist}

\usepackage[centering,margin=2.5cm]{geometry}

\newif\ifKfour 
\Kfourfalse

\newtheorem{theorem}{Theorem}[section]
\newtheorem{lemma}[theorem]{Lemma}

\newtheorem{obs}[theorem]{Observation}
\newtheorem{corollary}[theorem]{Corollary}
\newtheorem{proposition}[theorem]{Proposition}
\newtheorem{question}[theorem]{Question}

\theoremstyle{remark}

\theoremstyle{definition}

\newcommand\F{\mathcal{F}}

\newcommand\G{\mathcal{G}}

\newcommand\D{\mathcal{D}}

\newcommand\bb{\mathbf{b}}
\newcommand\cc{\mathbf{c}}
\newcommand\CC{\mathbf{C}}
\newcommand\K{\mathbb{K}}

\newcommand{\du}{\sqcup} 

\newcommand{\red}[1]{{\color{red}\bf #1}}

\newcommand{\blue}[1]{{\color{black}#1}} 

\DeclareMathOperator{\lk}{lk}

\DeclareMathOperator{\cl}{cl}
\DeclareMathOperator{\ico}{ico}
\DeclareMathOperator{\sign}{sign}

\title{On Betti numbers of flag complexes with forbidden
induced subgraphs}
\author{Karim Adiprasito\thanks{
Einstein Institute of Mathematics,
The Hebrew University of Jerusalem, Jerusalem, 91904 Israel. email: \url{adiprasito@math.huji.ac.il}.
Supported by ERC-2016-STG 716424 - CASe, NSF Grant DMS 1128155, and Israel
Science Foundation grant 1050/16.} \and
Eran Nevo\thanks{
Einstein Institute of Mathematics,
The Hebrew University of Jerusalem, Jerusalem, 91904 Israel. email: \url{nevo.eran@gmail.com}.
Partially supported by Israel Science Foundation grants ISF-805/11,
ISF-1695/15, by grant 2528/16 of the ISF-NRF Singapore joint research program, and
by ISF-BSF joint grant 2016288.
} \and Martin Tancer\thanks{Department of
Applied Mathematics, Charles University in Prague, Malostransk\'{e}
n\'{a}m\v{e}st\'{\i} 25, 118 00, Praha 1. email: \url{tancer@kam.mff.cuni.cz}.
Partially supported by the GA\v{C}R grant
16-01602Y and by Charles University project UNCE/SCI/004. Part of this work was done when M.~T. was affiliated with IST
Austria.}}
\date{\today}

\begin{document}
\maketitle
\begin{abstract}
We analyze the asymptotic extremal growth rate of the Betti numbers of clique complexes of graphs on $n$ vertices not containing a fixed forbidden induced subgraph $H$.

In particular, we prove a theorem of the alternative: for any $H$ the growth rate achieves exactly one of five possible exponentials, that is, independent of the field of coefficients, the $n$th root of the maximal total Betti number over $n$-vertex graphs with no induced copy of $H$ has a limit, as $n$ tends to infinity, and, ranging over all $H$, exactly five different limits are attained.

For the interesting case where $H$ is the $4$-cycle, the above limit is $1$, and we prove a superpolynomial upper bound.
\end{abstract}

\section{Introduction}
A central subject of extremal graph theory concerns {monotone} family of graphs
without a fixed subgraph, and its extremal properties -- starting with
Tur\'{a}n's theorem and the Erd\H{o}s-Stone theorem, \blue{on the maximal number of
edges in a graph not containing a fixed complete graph or complete multipartite graph
respectively -- as well as}
further generalizations
and refinements; \blue{see, e.g., \cite{Diestel}}.


The non-monotone family of
graphs $G$ without fixed \emph{induced} subgraphs have also been the subject of
extensive research \cite{CS}; for structure (e.g., perfect graphs, chordal graphs, coloring, \cite{K1}),
enumeration (e.g., \cite{Promel-Steger}), as well as extremal properties (e.g., Ramsey theory, Erd\H{o}s-Hajnal conjecture~\cite{Erdos-Hajnal,Chudnovsky--ErdosHajnal}).

Following Gromov and
subsequent work of Davis, Januszkiewicz and \'{S}wi\c{a}tkowski, the Betti
numbers of clique complexes without small induced cycles are central to the
study of nonpositive curvature in certain groups and manifolds; \blue{see
  \cite{MR2264834} and references therein. Januszkiewicz and
  \'{S}wi\c{a}tkowski~\cite{JS, MR2264834} used this connection to construct
  hyperbolic Coxeter groups of large cohomological dimension,
  which were long conjectured to be nonexistent by Bestvina, Gromov, Moussong
  and others, by constructing clique complexes without induced $4$-cycles and
  with high-dimensional cohomology. Simplifying and expanding on these constructions
  has since been an active topic, see also \cite{Osajda} for recent
developments.}

\blue{Here we focus on this fundamental problem from a different point of view.
First, we wish to understand the problem from a more quantitative perspective,
and understand how the topological complexity is in interplay with the size of
the complex as well as the forbidden substructure. Second, we wish to unify
perspectives of graph theory and geometric topology by studying not only the
case of clique complexes with forbidden cycles, but more general induced
subgraphs.}

\begin{question}
For any simple finite graph $H$,
what is the maximal total Betti number over all clique complexes $\cl(G)$ of
graphs $G$ with at most $n$ vertices and without an induced copy of $H$?
\end{question}
Let $\K$ be any field, $H$ be any simple finite graph, and
$$b_H(n) =
  b_H(n,\K)=\max_G\bigg\{\sum_{i\geq -1}\dim_\K
  \widetilde{H}_i(\cl(G);\K)\bigg\}$$
  where $G$ runs over all simple graph on at most $n$ vertices without an
  induced copy of $H$, and $\widetilde{H}_i(\cdot;\K)$ denotes the $i$th
  reduced homology with coefficients over $\K$. Note that $b_H(0)=1$ for any $H$, where $G=\emptyset$ is the only graph in the above $\max_G$.
We are interested in the growth of $b_H(n)$ as $n$ tends to infinity. The results turn out to be, quite interestingly, independent of the coefficient field.

Adamaszek~\cite{Adamaszek:Betti} showed that $b(n) \leq 4^{n/5}$, for
$$b(n)=\max_G\{\sum_{i\geq -1}\dim_\K \widetilde{H}_i(\cl(G);\K)\}$$
where $G$ runs over \emph{all} graphs on at most $n$ vertices.
Moreover the maximum is attained
by the complete multipartite graph $K_{5,5,\dots, 5}$ when n is divisible by $5$; we deduce that $\lim_{n\rightarrow \infty}\sqrt[n]{b(n)}=4^{1/5}$.

Therefore, if $H$ is not an induced subgraph of the infinite complete
multipartite graph $K_{5,5,\dots}$, then $b_H(n)$ may grow as quickly as
$(4^{1/5})^n$ (and again $\lim_{n\rightarrow \infty}\sqrt[n]{b_{H}(n)}=4^{1/5}$).

Thus, it is only interesting to study the function $b_H(n)$ for induced
subgraphs $H$ of $K_{5,5,\dots}$. The (finite) induced subgraphs of
$K_{5,5\dots}$ are exactly the complete multipartite graphs $K_{i_1,i_2, \dots,
i_m}$ where (without loss of generality) $5 \geq i_1 \geq \cdots \geq i_m \geq
1$. If $m=1$, then we get the independent set $I_{i_1}$ on $i_1$ vertices (it
should not be confused with $K_{i_1} = K_{1,\dots,1}$, the complete graph on
$i_1$ vertices, which is also an induced subgraph of $K_{5,5,\dots}$).

Adamaszek further showed that for $H=I_3$, the growth is exponential but
with a smaller base, at most $\approx 1.2499 < 4^{1/5} \approx 1.3195$.
It is also obvious that, if $H=K_d$ is a complete graph on $d$ vertices, then $\cl(G)$ is at most $(d-2)$-dimensional, and thus $b_{K_d}(n)=O(n^{d-1})$.

We will prove that the limit $\lim_{n\rightarrow
\infty}\sqrt[n]{b_{H}(n)}$ exists for any $H$ and we denote this
limit by $c_H$. Most strikingly, we will prove a theorem of the alternative: $c_H$, depending on $H$, can attain one of only 5 different values:

\begin{theorem}
\label{t:limit}
Let $H$ be any graph. The limit $c_H = \lim_{n\rightarrow
\infty}\sqrt[n]{b_{H}(n)}$ exists. In addition:
\begin{enumerate}[(i)]
\item
  If $H$ \blue{is not an induced subgraph of} $K_{5,5\dots}$, then $c_H = 4^{1/5} \approx 1.3195$.
\item
 For every $i \in \{1,\dots,5\}$ there is a value $c'_{i}$ with the following
 property. If $H = K_{i_1,\dots,i_m}$ with $5 \geq i_1 \geq \cdots \geq i_m\geq 1$, then
 $c_H = c'_{i_1}$. Moreover, $c'_5 = 3^{1/4} \approx 1.3161$, $c'_4 = 2^{1/3} \approx 1.2599$, $c'_3 \in [8^{1/14}, \Gamma_4] \approx
 [1.1601, 1.2434]$, and $c'_2 = c'_1 = 1$.

 Here $\Gamma_4$ is a certain constant which is precisely defined 
 in the Preliminaries.
\end{enumerate}
\end{theorem}

We summarize our results (including Adamaszek's bounds) in Table~\ref{tab:growth}.

\begin{table}[h!tb]
\small
  \begin{center}
  \begin{tabular}{lllll}
    & $H$ & $c_H$ & lower bound& upper bound \\
\toprule
& $H \not\leq K_{5,5,\dots}$ & $4^{1/5} \approx 1.3195$ & $4^{n/5}$ & $4^{n/5}$\\
\midrule
\multirow{2}{*}{$i_1 = 5$} & $I_5$ & \multirow{2}{*}{$3^{1/4} \approx 1.3161$} & $3^{n/4}$ & $3^{n/4}$\\
& $K_{5,\dots,5}$ ($m$ parts) & & $(4\cdot 3^{-5/4})^{m-1} 3^{n/4}$ & $(4\cdot 3^{-5/4})^{m-1} 3^{n/4}$\\
\midrule
$i_1=4$ &
$I_4$ & $2^{1/3} \approx 1.2599$ &
$2^{n/3}$
&$2^{n/3}$
\\
\midrule
\multirow{2}{*}{$i_1=3$} &
\multirow{2}{*}{$I_3$} & $\in [8^{1/14}, \Gamma_4] \approx$ &
\multirow{2}{*}{$8^{n/14}$}
& \multirow{2}{*}{$\Gamma_4^n$}\\
      & & $[1.1601, 1.2434]$ & & \\
\midrule
\multirow{4}{*}{$i_1 \leq 2$} &
$K_{2,2} = C_4$ & \multirow{4}{*}{$1$} & $\Omega(n^{3/2})$ &
\vphantom{$\int^{\binom{m}{n}}$} \blue{$n^{O({\log n})}$} \\
& $K_{2, 1, \dots, 1}$ ($m$ parts) & & $\Theta(n^{m-1})$ & $\Theta(n^{m-1}$) \\
& $K_{1, 1, \dots, 1}$ ($m$ parts) & & $\Theta(n^{m-1})$ & $\Theta(n^{m-1}$) \\
\bottomrule
\end{tabular}
\caption{The value $c_H$ and the upper and lower bounds on $b_H(n)$ for
interesting graphs $H$. The lower bounds are valid for infinitely many values
of $n$.}
\label{tab:growth}
\end{center}
\end{table}






Now, let us assume that $H = K_{i_1,\dots,i_m}$ is an induced subgraph of
$K_{5,5\dots}$ with $5 \geq i_1 \geq \dots \geq i_m$. Theorem~\ref{t:limit} shows that for $i_1 \in \{3,4,5\}$, the function $b_H(n)$ grows exponentially.
Let $H\leq G$ denote that $H$ is an induced subgraph of $G$.
The following theorem gives more refined bounds for any $I_5\leq H\leq K_{5,\dots,5}$.



\begin{theorem}\label{thm:I_3<H}
    If $H = K_{5,\dots,5}$ is $m$-partite, $m\geq 1$, then
    $$b_H(n)\leq  \left(\frac{4}{3^{5/4}}\right)^{m-1}\cdot 3^{n/4} \approx
    1.0131^{m-1} 1.3161^n.$$
    This bound is tight if $n - 5(m-1)$ is divisible by $4$ and positive, and is attained by the $(m-1)+\frac{n - 5(m-1)}{4}$-fold join consisting of $m-1$ copies of $I_5$ and the rest are $I_4$.
\end{theorem}


The upper bound given in Theorem~\ref{t:limit} for $H=I_3$ slightly improves the original bound by Adamaszek, but do not believe it to be optimal yet.
We present it mainly for the proof, which sets up a method how to push Adamaszek's approach further. We believe
that by the same method, the obtained value can be further improved, possibly even to the optimal bound, at the cost of a more extensive case analysis.

Regarding $H=I_4$, we show that $c'_4 = c_{I_4} = 2^{1/3}$.
The proof requires an extensive case analysis; therefore, we keep it separately
in the appendix. (However, some new ideas are needed as well to perform the analysis.)
In fact we show exact bound $b_{I_4}(n) \leq 2^{n/3}$; see
Theorem~\ref{t:K4} (in complementary setting, explained in the Preliminaries). This bound is tight if $n$ is
divisible by $3$, which is witnessed by the $n/3$-fold join of $I_3$. In this case,
we did not attempt to obtain a more precise bound for $H = K_{4,\dots,4}$ given
the length of the analysis for $I_4$.





We now improve the bounds for graphs where the growth is subexponential, specifically, for certain $H=K_{i_1,\ldots,i_m}$ where $i_j\leq 2$ for any $j$.
\begin{theorem}\label{thm:C_4}
  If $H=K_{2,2} = C_4$ is the $4$-cycle, then there are constants $c,C>0$ such that for any $n$
  \[cn^{3/2} < b_{C_4}(n) < \blue{n^{C{\log n}}}.\]
\end{theorem}
\begin{theorem}\label{thm:poly}
  If $H = K_{i_1,1,\dots,1}$ where $i_1 \leq 2$, then $b_H(n)$ has a polynomial
  growth \[b_H(n)=\Theta(n^{m-1})\] where $m$ is the number of parts in $H$.
\end{theorem}

Note that for $C_4=K_{2,2}$ our upper bound on $b_{C_4}(n)$ is subexponential but superpolynomial. The main problem on the growth of $b_{H}(n)$ that remains open is the following.

\begin{question}\label{q:poly}
For any $k\geq 2$ and $l\geq 0$ let $H=K_{2,\ldots,2,1,\ldots,1}$ with $k$ parts of size $2$ and $l$ parts of size $1$.
Does $b_H(n)$ have a polynomial growth, namely, is there a function $f(k,l)$ such that
$b_H(n) < n^{f(k,l)}$ for any large enough $n$?
\end{question}

A necessary condition for a superpolynomial growth when $H=C_4$ is that for any positive
integer $d$ there is a graph $G_d$ with no \blue{induced} $C_4$ such that $\cl(G_d)$
has a nonvanishing homology in dimension $>d$.
As mentioned, such constructions exist:
Januszkiewicz and \'{S}wi\c{a}tkowski \cite{JS} found $G_d$ such that $\cl(G_d)$ is a $d$-dimensional pseudomanifold
, for any positive integer $d$.

\blue{
Outline:
In Section \ref{sec:prelim} we overview relevant results of
Adamaszek~\cite{Adamaszek:Betti}, in Section \ref{sec:limit} we prove the existence of
the limit $c_H$, in Section \ref{sec:clique} we prove
Theorem~\ref{thm:poly},
in Section \ref{sec:C_4} we prove
Theorem~\ref{thm:C_4}, in Section \ref{sec:sup} we provide the exponential bounds for
$c_{I_5}, c_{I_3}$ stated in Theorem~\ref{t:limit} and the refined
bounds of Theorem~\ref{thm:I_3<H}. (Sections~\ref{sec:clique}, \ref{sec:C_4}
and~\ref{sec:sup} are mutually independent.) Concluding remarks are given in
Section~\ref{sec:conclude}. Appendix~\ref{a:K4} contains the proof of the
optimal bound for $c_{I_4}$.
}


\section{Preliminaries}\label{sec:prelim}

\blue{For technical reasons, it will be convenient for us to restate the main
results in `complementary setting', that is, we
consider here a clique complex over a graph as the independence complex over the complement of the graph. In particular, some of the graph
theoretical notions that we will meet along the way are much more natural in
the complementary setting. We will emphasize the complementary setting by bold
letters.}

\blue{
Let $G$ be a graph. By $\bar G$ we denote the complement of $G$. Next, by $\bb(G)$ we denote the sum
of the reduced Betti numbers of the independence complex of $G$ (computed over
some fixed field of coefficients). We also denote $\bb_H(n) := \max{\bb(G)}$
when the maximum is taken over all graphs on at most $n$ vertices without
induced copy of $H$. We have $\bb_H(n) = b_{\bar H}(n)$ because $H \not\leq G$
if and only if $\bar H \not\leq \bar G$. By $G_1 \du G_2$ we denote the
disjoint union of graphs $G_1$ and $G_2$. Note that the complement of the
infinite multipartite graph $K_{5,5,\dots}$ is the infinite disjoint union of
complete graphs $K_5
\du K_5 \du \cdots$. Similarly, the complement of $K_{i_1, \dots, i_m}$ is the
disjoint union of complete graphs $K_{i_1} \du \cdots \du K_{i_m}$.
}

\blue{
  Now, we may restate Theorems~\ref{t:limit} and \ref{thm:I_3<H} in the
  complementary setting (omitting approximations of the values). On the other
  hand, we do not restate Theorems~\ref{thm:C_4} and~\ref{thm:poly} as we prove
them in the primary setting.}

\blue{
  \begin{theorem}[Theorem~\ref{t:limit} in the complementary setting]
\label{t:limit_c}
Let $H$ be any graph. The limit $\cc_H = \lim_{n\rightarrow\infty}\sqrt[n]{\bb_{H}(n)}$ exists. In addition:
\begin{enumerate}[(i)]
\item
  If $H$ is not an induced subgraph of $K_5 \du K_5 \du \cdots$, then $\cc_H = 4^{1/5}$.
\item
 For every $i \in \{1,\dots,5\}$ there is a value $c'_{i}$ with the following
    property. If $H = K_{i_1} \du \cdots \du K_{i_m}$ with $5 \geq i_1 \geq \cdots \geq i_m\geq 1$, then
 $\cc_H = c'_{i_1}$. Moreover, $c'_5 = 3^{1/4}$, $c'_4 = 2^{1/3}$, $c'_3 \in [8^{1/14}, \Gamma_4]$, and $c'_2 = c'_1 = 1$.
\end{enumerate}
\end{theorem}
}

\blue{
  \begin{theorem}[Theorem~\ref{thm:I_3<H} in the complementary setting]\label{thm:I_3<H_c}
    If $H = K_5 \du \cdots K_5$ is a disjoint union of $m$ copies of $K_5$, $m\geq 1$, then
    $$\bb_H(n)\leq  \left(\frac{4}{3^{5/4}}\right)^{m-1}\cdot 3^{n/4}.$$
    This bound is tight if $n - 5(m-1)$ is divisible by $4$ and positive, and
    is attained by the $(m-1)+\frac{n - 5(m-1)}{4}$-fold disjoint union
    consisting of $m-1$ copies of $K_5$ and the rest are $K_4$.
\end{theorem}
}

\blue{
Now we overview some of the results of Adamaszek~\cite{Adamaszek:Betti} that
will be also useful for us.
Following Adamaszek, we keep presenting the results in the complementary setting.
We will occasionally need the following lemma, which easily follows from the
K\"{u}nneth formula.}



%

\begin{lemma}[{\cite[Lemma~2.1(a)]{Adamaszek:Betti}}]
  \label{l:du}
  Let $G$ and $H$ be two graphs. Then $\bb(G \du H) = \bb(G) \bb(H)$
\end{lemma}

Given a graph $G$, by the symbol $N[u] = N_{G}[u]$ we denote the
\emph{closed} neighborhood of a vertex $u$ in $G$, that is, the set of neighbors of $u$ including $u$.
Given a set $A$ of vertices of $G$, by
$G - A$ we mean the induced subgraph of
$G$ induced by $V(G)\setminus A$. We also write $G - v$ instead of $G - \{v\}$
for a vertex $v$ of $G$. Let us state another lemma by Adamaszek useful
for us.

\begin{lemma}[{\cite[Lemma~2.1(c)]{Adamaszek:Betti}}]
  \label{l:single_vertex}
  For any vertex $v$ of a graph $G$ we have $\bb(G) \leq \bb(G - v) +
  \bb(G - N[v])$.
\end{lemma}
The lemma follows from the Mayer-Vietoris long exact sequence for the decomposition of a simplicial complex as the union of a star and anti-star of some vertex.

Now, let us assume that $v$ is a vertex of degree $d$ of $G$ and let $v_1, \dots, v_d$ be all its neighbors (in arbitrarily chosen order).
An iterative application of the previous lemma gives the following recurrent
bound; see~\cite[Eq. (5)]{Adamaszek:Betti}. (Note that Adamaszek states the
bound in slightly different notation. He also assumes that $v$ is a vertex of
minimum degree. However, this assumption is unimportant in the proof of Eq.~(5)
in \cite{Adamaszek:Betti}; it is only used in subsequent computations.)

\begin{lemma}
\label{l:recurrent_simp}
Let $v$ be a vertex of degree $d$ and $v_1,\dots, v_d$ all its neighbors. Then
$$
\bb(G) \leq \sum\limits_{i=1}^{d} \bb(G - N[v_i] - \{v_1, \dots, v_{i-1}\}).
$$
\end{lemma}

From this lemma, Adamaszek deduces bounds on $\bb(G)$ for arbitrary graph $G$
and for a graph $G$ which is triangle-free. It is very useful for our further
approach to describe how to get such bounds from Lemma~\ref{l:recurrent_simp}.

Given a class $\G'$ of graphs, let $\bb(\G'; n)$ denote the maximum possible
$\bb(G')$ for a graph $G' \in \G'$ on at most $n$ vertices, assuming that such a graph exists (otherwise $\bb(\G'; n)$ remains undefined).
Let $\G$ denote the class
of all graphs and $\G_i$ denote the class of the $K_i$-free graphs, namely graphs with no copy of the complete graph on $i$ vertices.


From now on let us assume that $G$ is a fixed graph with $n$ vertices. We may
also assume that $G$ does not contain isolated vertices, otherwise $\bb(G) =
0$ \blue{(in this case, the independence complex of $G$ is a cone and therefore
contractible)}.
We also set $n_i$ to be the number of vertices of $G - N[v_i] - \{v_1,
\dots, v_{i-1}\}$ where $v$ and $v_1, \dots, v_d$ are as above, for $i \in
[d]$. In addition, \blue{from now on} we assume that $v$ is a vertex of minimum degree.
Lemma~\ref{l:recurrent_simp} implies
\begin{equation}
\label{e:rec_arbitrary}
\bb(G) \leq d \cdot \bb(\G; n - d -1)
\end{equation}
if $G$ is an arbitrary graph, and
\begin{equation}
\label{e:rec_triangle_free}
  \bb(G) \leq \sum\limits_{\blue{i=1}}^{\blue{d}} \bb(\G_3; n - i - d)
\end{equation}
if $G$ is triangle-free. Indeed, if $G$ is arbitrary, then $n_i$
is at most $n-(d+1)$ since $|N[v_{i+1}] \cup \{v_1,\dots, v_i\}| \geq
|N[v_{i+1}]| \geq d+1$,
and if $G$ is triangle-free, then $n_i$ is at most $n-(d+1+i-1)$ since $N[v_i]$ and $\{v_1, \dots, v_{i-1}\}$ are in addition disjoint.


In order to conclude a suitable bound on $\bb(G)$, it is sufficient to plug a
suitable function into the formulas above and prove the bound inductively.
Concretely, we set $\Theta_d = d^{1/(d+1)}$ and we set $\Gamma_d$ to be the
unique root on $[1,2]$ of the polynomial equation
\begin{equation}
\label{e:Gamma}
x^{2d} - x^{d-1} - x^{d-2} - \cdots - x - 1 = 0.
\end{equation}
It turns out that the sequence $\Theta_d$ is increasing on $[1,4]$ and
decreasing on $[4, \infty)$. In particular, it is maximized for $d=4$.
  Similarly, $\Gamma_d$ is increasing on $[1,3]$ and decreasing on $[3,
  \infty]$, therefore maximized for $d = 3$. Later on we will need to know
  approximative values of $\Gamma_d$ and $\Theta_d$ for small $d$; we provide
  these values for small $d$ in Table~\ref{tab:theta_gamma}.

\begin{table}
  \begin{center}
  \begin{tabular}{|c|c|c|c|c|c|}
\hline
$d$ & $1$ & $2$ & $3$ & $4$ & $5$ \\
\hline
$\Theta_d$ & 1 & 1.2599 & 1.3161 & 1.3195 & 1.3077  \\
\hline
$\Gamma_d$ & 1 & 1.2207 & 1.2499 & 1.2434 & 1.2293  \\
\hline
\end{tabular}
\caption{Approximative values of $\Theta_d = d^{1/(d+1)}$ and $\Gamma_d$.}
\label{tab:theta_gamma}
\end{center}
\end{table}

  Now, if we inductively assume that $\bb(\G; k) \leq \Theta_4^k$ for $k < n$,
  then Equation~\eqref{e:rec_arbitrary} gives
\begin{equation}
\label{e:theta_4_bound}
 \bb(G) \leq d \cdot \Theta_4^{n-d-1} = \Theta_4^n \cdot \frac{d}{\Theta_4^{d+1}}
 \leq \Theta_4^n \cdot \frac{d}{\Theta^{d+1}_d} = \Theta_4^n,
 \end{equation}
which proves $\bb(\G, n) \leq \Theta_4^n$. A similar computation yields
$\bb(\G_3, n) \leq \Gamma_3^n$.

 The first bound is tight as pointed out by Adamaszek, at least for $n$
 divisible by $5$. We will show that the second bound is not tight, and can be improved to $\Gamma_4^n < \Gamma_3^n$;
 see Section~\ref{sec:[3]}.

\section{Existence of the limit $\lim_{n\rightarrow \infty}\sqrt[n]{b_{H}(n)}$}\label{sec:limit}

Our \blue{first goal} is to show that the limit $\lim_{n\rightarrow \infty}\sqrt[n]{b_{H}(n)}$ exists for any $H$.

We \blue{still} work in the complementary setting. \blue{This means that we will prove
the existence of limit $\cc_H =\lim_{n\rightarrow\infty}\sqrt[n]{\bb_{H}(n)}$,
in the setting of Theorem~\ref{t:limit_c}.}
Recall from the Preliminaries that $\bb_H(n)=b_{\bar H}(n)$, where $\bar H$ is the complement of a graph $H$.

First, we consider the case when $H$ is connected.

\begin{proposition}\label{prop:limit}
If $H$ is connected, then the limit $\lim_{n\rightarrow \infty}\sqrt[n]{\bb_{H}(n)}$ exists.
\end{proposition}

\begin{proof}
Let, for any positive integer $n$, $G_n$ be a graph on at most $n$ vertices
which maximizes $\bb_{H}(n)$.

For any two positive integers $m,n$, the graph
$G_m \du G_n$ does not contain an induced copy of $H$ as $H$ is connected and
$G_m$ and $G_n$ do not contain an induced copy of $H$. (We recall that `$\du$'
stands for the disjoint union.) Therefore, by Lemma~\ref{l:du}, we get
$\bb_{H}(m+n) \geq \bb_H(m)\bb_{H}(n)$.
By the Fekete lemma for superadditive sequences~\cite{zbMATH02597740} (see also
\cite[Lem.11.6]{vLW}),
the limit $\lim_{n\rightarrow \infty}\sqrt[n]{\bb_{H}(n)}$ exists.
In addition, this limit is finite since we already know that
$\bb_H(n) \leq \Theta_4^n$ (or we can use the trivial bound $\bb_H(n) \leq
2^n$).
\end{proof}


We now turn to the case where $H$ is disconnected. \blue{Denote $\CC_H:=\limsup_n
\sqrt[n]{\bb_{H}(n)}$.}



\begin{proposition}\label{prop:limsup}
  Let $H_1$ and $H_2$ be any two graphs. Then $\CC_{H_1 \du H_2} \leq
  \max(\CC_{H_1}, \CC_{H_2})$. \blue{Consequently, the limit $\cc_H := \lim_{n\rightarrow
    \infty}\sqrt[n]{\bb_{H}(n)}$ exists for any graph $H$ and $\cc_{H_1 \du
    H_2} =
  \max(\cc_{H_1}, \cc_{H_2})$ for any two graphs $H_1$ and $H_2$.}
\end{proposition}


\begin{proof}

  \blue{We start by proving the first claim, the inequality, and
  will focus on the second claim, on $\cc_H$, at the end of the proof.}
  For
  simplicity of subsequent formulas, let $\alpha = \CC_{H_1 \du H_2}$ and
  $\alpha_i := \CC_{H_i}$ for $i \in \{1,2\}$. Without loss of generality, we
  will assume that $\alpha_1 \geq \alpha_2$, that is, our task is to show that
  $\alpha \leq \alpha_1$. We will achieve this task by showing that $\alpha
  \leq \alpha_1 + \varepsilon$ for any $\varepsilon > 0$.


  Form now on, let us fix $\varepsilon > 0$. We also fix a large enough integer
  parameter $p$ which depends on $\varepsilon$, but we will describe the exact
  dependency later on. Now let $G_n$ be a graph on at most $n$ vertices which
  maximizes $\bb_{H_1 \du H_2}(n)$, in particular, it does not contain an
  induced copy of $H_1 \du H_2$. By the definition of $\alpha_1$ we get
\begin{equation}
\label{e:bound_H1}
  \bb_{H_1}(n) \leq k(\varepsilon)(\alpha_1 + \varepsilon)^n
\end{equation}
 for every $n$
  where $k(\varepsilon)$ is a large enough constant depending only on
  $\varepsilon$. \blue{(From the definition of $\limsup$, we get $\bb_{H_1}(n) \leq
  (\alpha_1 + \varepsilon)^n$ for large enough $n$, depending on $\varepsilon$.
  The purpose of $k(\varepsilon)$ is to ensure validity of the
  inequality~\eqref{e:bound_H1} for all
$n$.)} Since $\alpha_2 \leq \alpha_1$, we can also assume that
\begin{equation}
\label{e:bound_H2}
  \bb_{H_2}(n) \leq k(\varepsilon)(\alpha_1 + \varepsilon)^n
\end{equation} eventually by adjusting $k(\varepsilon)$.

  Our aim is to show by induction that
\begin{equation}
\label{e:ind}
  \bb(G_n) \leq
  2^p k(\varepsilon)(\alpha_1 + \varepsilon)^n.
\end{equation}
  Note that this inequality is true for $n = 1$ since $\alpha_1 \geq 1$.

  It remains to prove Eq.~\eqref{e:ind} for a fixed $n$ assuming that it is
  true for every smaller value. Let us distinguish several cases.

\medskip

In the first case we assume that $G_n$ does not contain an induced copy of
$H_2$. Then we get the desired inequality directly from Eq.~\eqref{e:bound_H2}.

\medskip


In the second case, let us assume that there are at most $p$ vertices
of $G_n$ such that when we remove these vertices, we get a graph which does not
contain an induced copy of $H_1$. Our next task is to show that in this case,
$\bb(G_n) \leq 2^pk(\varepsilon)(\alpha_1 + \varepsilon)^n$ which implies
desired Eq.~\eqref{e:ind}. Let $u_1, \dots, u_j$, \blue{$j \leq p$}, be the removed vertices from
$G$. Let \blue{$G'$} be any induced subgraph of $G$ and let
$i := |V(\blue{G'}) \cap \{u_1, \dots, u_j\}|$ be the number of vertices
$u_1, \dots, u_j$ in \blue{$G'$}. We will prove by induction in $i$ that
\begin{equation}
\label{e:ind_i}
\blue{\bb(G')} \leq 2^ik(\varepsilon)(\alpha_1 + \varepsilon)^n.
\end{equation}
When we specify Eq.~\eqref{e:ind_i} to $G$, that is, \blue{$i = j \leq p$}, we get the desired
inequality.

The first induction step for $i = 0$ follows from the fact that \blue{$G'$} is
$H_1$-free in this case and from Eq.~\eqref{e:bound_H1}. (We could get a better
bound since the number of vertices of \blue{$G'$} is (typically) less than $n$, but
we do not need such an improvement.)

The second induction step for $i > 0$ follows directly from
Lemma~\ref{l:single_vertex} by removing one of the vertices $u_1, \dots, u_j$
which is also a vertex of \blue{$G'$}.

\medskip


Finally, we distinguish a third case when we assume that $G_n$ contains an
induced copy of $H_2$ and after removing any $p$ vertices from $G_n$ we still
get a graph that contains an induced copy of $H_1$. Let $H'_2$ be an induced
copy of $H_2$ in $G$ and let $h_i$ be the number of vertices of $H_i$ for $i
\in \{1,2\}$. We will prove that $H'_2$ contains a vertex of degree at least
$\frac{p-h_2}{h_2}$ in $G_n$. It is sufficient to show that there are more than
$p - h_2$ edges connecting $H'_2$ and the remainder of $G$. For contradiction,
\blue{suppose} there are at most $p - h_2$ such edges. Let $H'$ be the induced subgraph of
$G$ consisting of $H_2$ and all neighbors of vertices of $H_2$ inside $G$. Then
$H'$ has at most $p$ vertices. Consequently, there is an induced copy $H'_1$ of $H_1$ inside the induced subgraph of $G_n$ obtained from $G_n$ by removing the
vertices of $H'$ by our assumption of this distinguished case. By the definition
of $H'$, the two copies $H'_1$ and $H'_2$ are connected by no edge and
therefore we have found an induced copy of $H_1 \du H_2$, this is
a contradiction.

$H'_2$ contains a vertex $v$ of degree
$d \geq \frac{p-h_2}{h_2}$. Lemma~\ref{l:single_vertex} gives
\[
\bb(G_n) \leq \bb(G_n - v) + \bb(G_n - N_{G_n}[v])).
\]
Note that $G_n - v$ has at most $n-1$ vertices, $G_n - N_{G_n}[v]$ has at most
$n - d - 1$ vertices and both these graphs do not contain an induced copy of
$H_1 \du H_1$ since they are induced subgraphs of $G_n$. Therefore, the
induction in $n$ gives us

$$
\bb(G_n) \leq 2^p k(\varepsilon) \left( (\alpha_1 + \varepsilon)^{n-1} +
(\alpha_1 + \varepsilon)^{n-d -1} \right).
$$

It is easy to check that
$$
(\alpha_1 + \varepsilon)^{-1} + (\alpha_1 + \varepsilon)^{- d - 1} \leq 1
$$
if $d$ is large enough, that is, if $p$ is large enough, for fixed
$\varepsilon$, since $\alpha_1 \geq 1$.
Combining the two above-mentioned inequalities, we get the desired
inequality~\eqref{e:ind}.

\blue{It remains to deduce that the limit $\cc_H := \lim_{n\rightarrow
  \infty}\sqrt[n]{\bb_{H}(n)}$ exists for any graph $H$ and $\cc_{H_1 \du
    H_2} =
  \max(\cc_{H_1}, \cc_{H_2})$ for any two graphs $H_1$ and $H_2$.}

\blue{
We get
$$
\liminf_{n\rightarrow
    \infty}\sqrt[n]{\bb_{H_i}(n)}\leq \liminf_{n\rightarrow
        \infty}\sqrt[n]{\bb_{H_1 \du H_2}(n)}
$$
for any two graphs $H_1$, $H_2$ and $i \in \{1,2\}$, as a graph without induced
copy of $H_i$ does not contain induced copy of $H_1 \du H_2$. Consequently
$$
\max_{i=1,2}\liminf_{n\to\infty} \sqrt[n]{\bb_{H_i}(n)}\leq
\liminf_{n\rightarrow
          \infty}\sqrt[n]{\bb_{H_1 \du H_2}(n)} \leq
\CC_{H_1 \du H_2} =
\max_{i=1,2}(\CC_{H_1}, \CC_{H_2}).
$$
If $H = H_1 \du H_2$ is a disjoint union of two connected graphs $H_1$ and
$H_2$, then $\cc_{H_1}$ and $\cc_{H_2}$ exist by Proposition~\ref{prop:limit}.
Therefore all inequalities above are equalities and $\cc_H =
\max(\cc_{H_1},\cc_{H_2})$ exists. The general case of $H$ with multiple
components follows analogously via a simple induction on the number of components.
}
\end{proof}


\blue{
From Proposition~\ref{prop:limsup}, we also conclude that $\cc_H =
\cc_{K_{i_1}}$ for $H = K_{i_1} \du \cdots \du K_{i_m}$ as in the statement of
Theorem~\ref{t:limit_c}(ii), as for $l>m$ a graph with no $K_m$ clearly has no $K_l$ as induced subgraph. This also implies the existence of the values
$c'_{i_1}$.
}

\blue{
To finish the proof of Theorem~\ref{t:limit}(ii) we need to find or bound the
constants $c'_3$, $c'_4$ and $c'_5$ (clearly $c'_i = \cc_{K_i} = 1$, for $i \in
{1,2}$, realized by the empty graph). This we do in Section~\ref{sec:sup}.
}

%

\section{Polynomial growth for $H\leq K_{2,1,\ldots,1}$}\label{sec:clique}
We prove Theorem~\ref{thm:poly}. Clearly, if $H=K_d$ is a clique of size $d$, then all faces of $\cl(G)$ have
size $\leq d-1$ and so $b_{K_d}=O(n^{d-1})$. \blue{On the other hand, let us
  consider the Tur\'{a}n graph $T_{d,n}$ on $n$ vertices and without a $K_d$.
  This means that $T_{d,n} = K_{n_1, \dots, n_{d-1}}$ where $n_1, \dots,
  n_{d-1}$ are as equal as possible under the condition $n_1 + \cdots +
  n_{d-1} = n$. We have that $\cl(T_{d,n})$ has dimension $d-2$. Once we fix a
  vertex $v_i$ in the $i$th part of $T_{d,n}$ for $i \in \{1,\dots,d-1\}$,
  the $(d-2)$-faces avoiding the vertices $v_i$} form a basis of $H_{d-1}(\cl(T_{d,n});k)$; there are $\Omega(n^{d-1})$ such faces, thus \[b_{K_d}=\Theta(n^{d-1}).\]


The other case left to consider is $H=K_{d+1}^-$, the complete graph on $d+1$ vertices minus one edge. Then, any two simplices of $\cl(G)$ of dimension $\geq d-1$ intersect in a face of dimension at most $d-3$. Thus, an iterated application of the Mayer-Vietoris sequence (or one application of the Mayer-Vietoris spectral sequence) shows that the union of all simplices of dimension $\geq d-1$ in $\cl(G)$ is a complex with vanishing homology in dimensions $\geq d-1$. Thus, the entire complex $\cl(G)$ has vanishing homology in dimensions $\geq d-1$, and so $b_{K_{d+1}^-}(n)=O(n^{d-1})$.
As $K_d\leq K_{d+1}^-$ the lower bound provided by the Tur\'{a}n graph $T_{d,n}$ applies, and we conclude
\[b_{K_{d+1}^-}=\Theta(n^{d-1}).\]

\section{Subexponential growth for $H=C_4$}\label{sec:C_4}
We now prove Theorem~\ref{thm:C_4}. We start with the upper bound:
\begin{theorem}\label{thm:C_4UBT}
 $b_{C_4}(n)<n^{O(\log n)}$.
\end{theorem}

\begin{proof}
We show that if $G$ has no induced $C_4$ and $\cl(G)$ has nontrivial homology in dimension $d$, then $G$ must have many vertices, specifically at least 
\blue{$2^d$}.

Given a homology $d$-cycle $z$ in $\cl(G)$, let $z_0$ denote 
the vertex support of $z$.
For a subset $A$ of the vertices of $G$ let $G[A]$ denote the induced subgraph on $A$, and
$\cl(A):=\cl(G[A])$ for short. Let $N_G(v)$ denote the set of neighbors of $v$
in $G$ (excluding $v$).
For a vertex $v$ in $ z_0$ \blue{denote by} $\lk_{z}(v)$ the $(d-1)$-chain in
the complex $\cl(N_z(v))$ induced by the link map, where $N_z(v)$ is the set
$N_G(v)\cap z_0$; \blue{namely, for a homology $d$-cycle
  $z=\sum\alpha_{\sigma}\sigma$ where the sum is taken over all $d$-faces
$\sigma$ in $\cl(G)$, we have $\lk_{z}(v):=\sum \alpha_{\sigma} \sign
(v,\sigma) \sigma\setminus\{v\}$ where the sum is taken over all $d$-faces
$\sigma$ in $\cl(G)$ containing $v$.} \blue{Clearly}, $\lk_{z}(v)$ is a
cycle\blue{; indeed, the coefficient of any $(d-2)$-face
$\sigma\setminus\{v,u\}$ in the boundary $\partial \lk_{z}(v)$ equals the
coefficient of the $(d-1)$-face $\sigma\setminus\{u\}$ in $\partial z$, which
is zero}.

Now, define the following \blue{two} functions:
\begin{itemize}
\item $\beta(d)$ is the minimal number of vertices in the support of a nontrivial homological $d$-cycle $ z$ in $\cl(G)$, over all graphs $G$ \blue{with no induced $C_4$}; we will show that $\beta(d)\geq \blue{2^d}$.
\blue{
\item $\gamma(d)$ is the minimal cardinality of the set of vertices
$z_0\setminus (\{v\}\cup N_G(v))$ over all triples $(G,z,v)$ where $z$ is a nontrivial homological $d$-cycle in $\cl(G)$, $v$ is a vertex in $G$ and the graph $G$ has no induced $C_4$.
} 
\end{itemize}

\blue{We will show $\gamma(d) \geq 2 \gamma(d-1)$. Together with $\gamma(1) = 2$,
this gives $\beta(d) \geq \gamma(d) \geq 2^d$, and hence
$b_{C_4}(n)\leq n^{O(\log n)}$ as required.}

\blue{
It remains to show $\gamma(d) \geq 2 \gamma(d-1)$.
Let the triple $(G,z,v)$ realize $\gamma(d)$, thus the number of vertices in
$z_0\setminus (\{v\}\cup N_G(v))$ equals $\gamma(d)$. Choose $G$ to have the
minimal number of vertices, running over all such triples; in particular
$z_0\cup \{v\}$ is the vertex set of $G$.}

\blue{
First, we claim that for any vertex $u\in z_0\cap N_G(v)$, the $(d-1)$-cycle
$\lk_z(u)$ is nontrivial in $\cl(G(u))$ for $G(u)$ the induced subgraph $G[\lk_z(u)_0\cup\{v\}]$ of $G$.
Indeed, otherwise consider a $d$-chain $b$ that bounds $\lk_z(u)$ in $\cl(G(u))$. The $d$-cycle $z'=z-u*\lk_z(u)+b$ is homologous to $z$ in $\cl(G)$, hence the triple $(G-u,z',v)$ also realizes $\gamma(d)$, contradicting the minimality of the order of $G$.
We conclude that for every $u\in z_0\cap N_G(v)$ (if exists), the number of
vertices in $\lk_z(u)_0\setminus (\{v\}\cup N_{G(u)}(v))$ is at least
$\gamma(d-1)$.}

\blue{Next, we claim that there are at least two different vertices $u,w\in
  z_0\cap N_G(v)$ such that $uw$ is not an edge in $G$ (for $d\ge 1$). That
  $|z_0\cap N_G(v)|\ge d+1\ge 2$ follows from the definition of $\gamma(d)$ by
  considering any vertex in $z$. Assume by contradiction that $z_0\cap N_G(v)$
  forms a clique, and consider any vertex $v'\in z_0\cap N_G(v)$. The $d$-cycle
  $z'=z-v*\lk_z(v)+v'*\lk_z(v)$ in case $v\in z_0$, and $z'=z$ if $v\notin
  z_0$, is homologous to $z$, so the triple $(G-v,z',v')$ must also realize
$\gamma(d)$, and contradicts the minimality of the order of $G$.}




\blue{Note that as $G$ has no induced $C_4$, the sets 
$\lk_z(u)_0\setminus (\{v\}\cup N_{G(u)}(v))$
and $\lk_z(w)_0\setminus (\{v\}\cup N_{G(w)}(v))$ 
are disjoint;
combining with the last two claims gives 
\[\gamma(d)=|z_0\setminus (\{v\}\cup N_G(v))|\ge 
|\lk_z(u)_0\setminus (\{v\}\cup N_{G(u)}(v))|+
|\lk_z(w)_0\setminus (\{v\}\cup N_{G(w)}(v))|\ge
2\gamma(d-1).\]}
\end{proof}

We now turn to the lower bound. \blue{For any prime $p$, let $G_p$ denote the}
\blue{incidence} graph of the \blue{finite} projective plane of order $p$, \blue{that is, its vertices correspond to the $1$- and $2$-dimensional linear subspaces of $\mathbb{Z}_p^3$ and its edges correspond to strict containment relations between them.
Note that $G_p$ is}
bipartite, connected, with no $C_4$, it has
$v/2=p^2+p+1$ vertices on each side and $e=(p+1)(p^2+p+1)$ edges;
\blue{see for example~\cite[Example 19.7]{vLW} or~\cite{Matousek-Nesetril}.}
Thus, its
total Betti number equals the dimension of the first homology group, which
equals $e-v+1$ \blue{(this is the number of edges outside an arbitrary spanning
tree)}.  Given $n$, add some isolated vertices to the above graph where
$p$ is the largest prime for which $n/2 \geq p^2+p+1$, to obtain a graph $G$
with $n$ vertices. As clearly $n/2 <(2p)^2+2p+1$, we conclude that $\cl(G)$ has
total Betti number of order $\Omega(n^{3/2})$. Thus:

\begin{corollary}\label{cor:C_4}
There exists a \blue{positive} constant $c$ such that for any $n$, $b_{C_4}(n)>cn^{3/2}$.
\end{corollary}

Combining Theorem~\ref{thm:C_4UBT} and Corollary~\ref{cor:C_4} gives Theorem~\ref{thm:C_4}.

\section{Comparing the exponential growth for graphs $I_3 \leq H \leq K_{5,5,\dots}$}\label{sec:sup}


In this section we again work in complementary setting, as described in the
Preliminaries. \blue{This means that the estimates on $b(H)$ for $I_3 \leq H
\leq K_{5,5,\dots}$ translate as estimates on $\bb(H)$ for $K_3 \leq H \leq K_5
\du K_5 \du \cdots$. That is, we focus on the graphs $H = K_{i_1} \du \cdots
\du K_{i_m}$ with $i_1 \geq i_2 \geq \cdots \geq i_m$ and $i_1 \in \{3,4,5\}$.
(The case $i_1=4$ is deferred to Appendix A.)}

\subsection{$K_5$-free graphs}

Recall \blue{from the Preliminaries} that $\G_5$ denotes the class of $K_5$-free graphs, namely, it
consists of the graphs with no induced $K_5$. In this case, the upper bound on the homology growth can be improved from $\Theta_4^n$ to $\Theta_3^n$, which is tight.

\begin{proposition}
\label{p:no_K5}
$\bb(\G_5;n) \leq \Theta_3^n$.
\end{proposition}

%

\begin{proof}
Let $G$ be a $K_5$-free graph with $n$ vertices.
%
The proof is by induction on $n$. The base case $n=0$ trivially holds as
$\bb(\emptyset)=1$. We may also assume that $G$ does not contain an isolated
vertex otherwise $\bb(G) = 0$.

Let $d$ be the minimum degree of $G$. If $d \neq 4$, then the same computation
as in Equation~\eqref{e:theta_4_bound} gives
(note that
$\Theta_3 \geq \Theta_d$ if $d \neq 4$; see Table~\ref{tab:theta_gamma}):
$$
  \bb(G) \leq \blue{d \cdot } \Theta_3^{n-d-1} =
\Theta_3^n\frac{d}{\Theta_3^{d+1}} \leq
\Theta_3^n\frac{d}{\Theta_d^{d+1}} = \Theta_3^n.
$$

It remains to consider the case $d = 4$. That is, $v$ has neighbors $v_1,
\dots, v_4$. As $G$ is $K_5$ free, there is at least one missing edge among
these neighbors. For simplicity, we can assume that this missing edge is $v_1v_2$
since we can choose the order of the neighbors of $v$.
We deduce that $n_i \leq n - d - 1 = n -
5$ for $i \in \{1, 2, 3, 4\}$
as usual, where $n_i$ is the number of vertices of
$G - N[v_i] - \{v_1,
\dots, v_{i-1}\}$.
 However, in addition, we can deduce that $n_2 \leq n-6$ because
$N[v_2]$ does not contain $v_1$. Therefore,
Lemma~\ref{l:recurrent_simp} gives (using that $f(n)=\Theta^n$ is increasing)
$$
\bb(G) \leq 3\Theta_3^{n-5} + \Theta_3^{n-6} = \Theta_3^n
\Theta_3^{-6}(3\Theta_3 + 1).
$$

Now, the equation $3x + 1 = x^6$ has a root $x_0 \doteq 1.3038 < \Theta_3$ (this is the
only root on $[1,2]$), and therefore it is easy to deduce that $3\Theta_3 + 1
\leq \Theta_3^6$ as $\Theta_3 \geq x_0$ (one can also put directly $3\Theta_3 +
1$ and $\Theta_3^6$ into a calculator).
This gives the desired bound $\bb(G) \leq \Theta_3^n$.
\end{proof}

The bound provided by Proposition~\ref{p:no_K5} is tight for $n$ divisible by $4$, as the disjoint union of $n/4$ copies of $K_4$ shows. For $n$ not divisible by $4$ change the sizes of one or two components such that each of them have size $>1$, to conclude $\bb(\G_5;n)\geq \frac{2}{9}\Theta_3^n$.
Thus, we get the following corollary; \blue{following the notation of
Theorem~\ref{t:limit_c}(ii).}
\begin{corollary}\label{cor:I_5}
  $\blue{c'_5 = \cc_{K_5}(n)=\Theta_3}.$
\end{corollary}

\subsection{$mK_5$-free graphs}
Let $mK_5$ denote the disjoint union of $m$ copies of $K_5$.
By Theorem~\ref{t:limit}(ii) and Proposition~\ref{p:no_K5} we already know that
for any \blue{$I_5\leq H\leq K_{5,5,\dots}$}, $c_H=\Theta_3$. Here we refine
the upper bound on $\bb_{mK_5}(n)$, as asserted in
\blue{Theorem~\ref{thm:I_3<H_c}}.

\begin{proof}[Proof of Theorem~\ref{thm:I_3<H_c}]
We prove the result by a double induction. The outer induction is in $m$, the
inner induction is in $n$. The case $m = 1$ was proved in the previous \blue{sub}section,
thus we can assume $m \geq 2$.

First, let us assume that $G$ contains $k$ isolated copies of $K_5$ for some
$k > 0$. Let $G'$ be $G$ without these
 copies. Note that $k \leq m-1$ and that $G'$ is $(m-k)K_5$-free.
  Then
 $$
 \bb(G) = \bb(G')\bb(K_5)^k \leq \left(4^{m-k-1}\Theta_3^{n - 5k -
 5(m-k-1)}\right) \cdot 4^k
 $$
 where the equality follows from Lemma~\ref{l:du}
 and the
inequality follows from the induction and from $\bb(K_5) = 4$. That is, $\bb(G)
\leq 4^{m-1}\Theta_3^{n-5(m-1)}$ as desired.

 If $G$ does not contain an isolated copy of $K_5$ then we proceed analogously
 as in the previous subsection. We let $d$ be the minimum degree of $G$ and we
 consider a vertex $v$ of degree $d$ and its neighbors.

 If $d \neq 4$, then Lemma~\ref{l:recurrent_simp} implies
 \begin{align*}
   \bb(G) &\leq d\cdot 4^{m-1} \Theta_3^{n-d-1 - 5(m-1)} =
4^{m-1} \Theta_3^{n-5(m-1)}\frac{d}{\Theta_3^{d+1}}\\
&\leq
4^{m-1} \Theta_3^{n-5(m-1)}\frac{d}{\Theta_d^{d+1}} = 4^{m-1}
\Theta_3^{n-5(m-1)}.
\end{align*}


Now, let us assume that $d = 4$. Since we assume that $G$ has no isolated
$K_5$, we either miss some edge among
the neighbors $v_1, \dots, v_4$, or the degree of some of the vertices $v_1,
\dots, v_4$ is greater than $4$. In both cases, Lemma~\ref{l:recurrent_simp}
provides us with a bound

\begin{align*}
  \bb(G) &\leq 3\cdot4^{m-1}\Theta_3^{n-5-5(m-1)} + 4^{m-1}\Theta_3^{n-6-5(m-1)} = 4^{m-1}\Theta_3^{n-5(m-1)}
\Theta_3^{-6}(3\Theta_3 + 1) \\
&\leq 4^{m-1}\Theta_3^{n-5(m-1)}
\end{align*}
as wanted. (Here we again use the inequality $3\Theta_3 + 1 \leq \Theta_3^6$ explained at the end of the proof of
Proposition~\ref{p:no_K5}.)
\end{proof}

\subsection{$K_3$-free graphs}\label{sec:[3]}

We recall that it was explained in the Preliminaries how to get Adamaszek's bound
\blue{$c'_3 = \cc_{K_3} = c_{I_3}\leq \Gamma_3$}.
We aim to get an improved bound \blue{$c'_3 \leq \Gamma_4$}.
The idea behind the improvement is that a more detailed combinatorial analysis
of $N[v_i]$, in the setting of Lemma~\ref{l:recurrent_simp}, reveals one of the following three options. Either $d
\neq 3$ and we can use a bound with $\Gamma_4$, or $d = 3$ and $v$ can be
chosen so that some of the neighbors of $v$ has degree at least $4$ which again
improves the bound, or, finally (assuming connectedness), $G$ is a cubic graph
which means that $N[v_i]$ is a $2$-degenerate graph, which again
yields improving the bound.

Before addressing general triangle free graphs,
 it is useful first to give an upper bound on $\bb(G)$ for triangle-free graphs $G$ which are $2$-degenerate.

\subsubsection{$2$-degenerate triangle free graphs}
Let $\D_k$ be the class of $k$-degenerate graphs, that is graphs,
 such that for every $G$ in $\D_k$ and for
 every (induced) subgraph $G'$ of $G$, the minimum degree of $G$ is at most $k$.

 \begin{proposition}
   \label{p:D2}
   Let $G \in \D_2$ be a triangle-free graph on $n$ vertices. Then
$$
\sqrt[n]{\bb(G)} \leq \Gamma_2 \doteq 1.2207.
$$
 \end{proposition}

The bound $\Gamma_2$ is very probably not an optimal one in this case. However,
it is sufficient for our purposes.

\begin{proof}
  The proof is essentially the same as the proof of Adamaszek's bound for triangle free graphs
  using, in addition, the fact that the minimum degree is at most $2$.
Assume $G$ has no isolated vertex, else the assertion is trivial, as $\bb(G)=0$ in this case.

  Let $v$ be a vertex of minimum degree $d$ and $v_1$ and $v_2$ (or just $v_1$)
  be its neighbors. If $d = 2$, Lemma~\ref{l:recurrent_simp} yields
  $$
  \bb(G) \leq \Gamma_2^{n-3} + \Gamma_2^{n-4} =
  \Gamma_2^n\Gamma_2^{-4}(\Gamma_2 + 1) = \Gamma_2^n.
  $$
In the induction, we crucially use that the subgraphs
$G - N[v_1]$ and $G - N[v_2] - v_1$ are also triangle-free
 graphs in $\D_2$.


  If $d = 1$, we even get $\bb(G) \leq \Gamma_2^{n-2} < \Gamma_2^n$ from
  Lemma~\ref{l:recurrent_simp}.
\end{proof}

  \subsubsection{General triangle-free graphs}
Here we prove the promised bound, namely,
\begin{proposition}
Let $G$ be a triangle-free graph on $n$ vertices. Then
$$
\bb(G) \leq \Gamma_4^n.
$$
 \end{proposition}

\begin{proof}
As usual, the proof is by induction on $n$, again $d$ is the minimum degree, $v$ is a vertex of the minimum degree and $v_1, \dots, v_d$ are its neighbors.

First, we can assume that $G$ is connected. Indeed, if $C_1, \dots, C_k$ are
the components of $G$ then we can deduce $\bb(G) \leq \Gamma_4^n$ from $\bb(G) =
\bb(C_1)\cdots\bb(C_k)$ (see Lemma~\ref{l:du}) and from the induction.

If $d \neq 3$, then we deduce
$$
\bb(G) \leq \Gamma_4^n
$$
from induction analogously to the computations in the proof of Proposition~\ref{p:no_K5}.
Indeed
  \blue{
$$
\bb(G) \leq \sum\limits_{i=1}^{d}\Gamma_4^{n-i-d} =
\Gamma_4^n\sum\limits_{i=1}^{d} \Gamma_4^{-i-d} \leq
\Gamma_4^n\sum\limits_{i=1}^{d} \Gamma_d^{-i-d}=
\Gamma_4^n \Gamma_d^{-2d}(1 + \Gamma_d + \cdots + \Gamma_d^{d-1}) =
\Gamma_4^n.
$$
}
The first inequality follows from the induction analogously to
Eq.~\eqref{e:rec_triangle_free}. The last equality follows from the definition
of $\Gamma_d$ \blue{via~\eqref{e:Gamma}}.
Also note that $\Gamma_4$ is the largest value among $\Gamma_d$, with $d \neq 3$.

It remains to consider the case $d=3$.
We will distinguish two subcases.

In the first subcase, $G$ is not a cubic graph ($3$-regular). That means, it
contains two vertices, one of them of degree $3$ and the second one of degree
greater than $3$. Thus we can adjust our choice of $v$ and its neighbors $v_1,
v_2, v_3$ so that the degree of $v_1$ is at least $4$. This means that $n_1
\leq n-5$, $n_2 \leq n-5$, and $n_3 \leq n-6$, as there is no edge between
$v_1,v_2,v_3$. Lemma~\ref{l:recurrent_simp}
now gives a bound
$$
\bb(G) \leq \Gamma_4^n(\Gamma_4^{-5} + \Gamma_4^{-5} + \Gamma_4^{-6}) =
\Gamma_4^n \Gamma_4^{-6}(2\Gamma_4 + 1)
.
$$
The equation $2x + 1 = x^6$ has a unique solution $x_1 \doteq 1.2298$ on
$[1,2]$ and we can deduce that $\bb(G) \leq \Gamma_4^n$ since $\Gamma_4 \geq
x_1$; see Table~\ref{tab:theta_gamma}.

In the second subcase we assume that $G$ is a (connected) cubic graph.
In this subcase, we will not save the value on the exponents, but we will save it on the bases.
More concretely, in this case we crucially use that the graphs $G_i -
N[v_i] - \{v_1,\dots,v_{i-1}\}$ belong to $\D_2$ since they are proper subgraphs of a
connected cubic graph. Therefore, we can use Proposition~\ref{p:D2} and
together with Lemma~\ref{l:recurrent_simp}, for $n \geq 7$, we deduce
$$
\bb(G) \leq \Gamma_2^n(\Gamma_2^{-4} + \Gamma_2^{-5} + \Gamma_2^{-6}) \leq
\Gamma_4^n.
$$
It is easy to check that for $n \leq 6$, the only possible cubic triangle-free
graph is $K_{3,3}$. In this case $\bb(K_{3,3}) = 1$ and the required inequality is satisfied as well.
\end{proof}

\section{Concluding remarks}\label{sec:conclude}

As mentioned in the Introduction, we still do not know whether there exists a graph $H$ for which $b_H(n)$ grows subexponentially and superpolynomially. See  Question~\ref{q:poly} for the candidates for such $H$.



The computation of $b_H(n)$ reduces to graphs with \emph{exactly} $n$ vertices:

\emph{Monotonicity}.
By definition, for any graph $H$ clearly $b_H(n)$ is weakly increasing.
Let $b^=_H(n)$ be the maximum total Betti number among all graphs with no induced copy of $H$ and with \emph{exactly} $n$ vertices. In fact,
\begin{obs}\label{q:monotonicity}
For any graph $H$, the function $b^=_H(n)$ is weakly increasing in $n$.
\end{obs}
\begin{proof}
First note that when adding to $G$ an isolated vertex $v$, the total Betti number of $\cl(G\sqcup v)$ is one more than of $\cl(G)$, where the $0$th Betti number is increased by one. Thus, the result holds for $H$ with no isolated vertex. Next, for $H=H'\sqcup u$, $G$ a maximizer of $b^=_H(n)$ and a vertex $w\in G$, let $G'$ be obtained from $G$ by adding a new vertex $v$ and connecting it to $w$ and all neighbors of $w$. Then $\cl(G')$ deformation retracts to $\cl(G)$, so they have the same total Betti number.
If $H\leq G'$ then any induced copy of $H$ in $G'$ must contain $v$ and $w$; but then for $G"=G\sqcup v$ we get $H\nleqslant G"$, and again the total Betti number of $\cl(G")$ is one more than of $\cl(G)$.
\end{proof}

\section*{Acknowledgment}
We are very thankful to the anonymous referee for many valuable remarks.

\bibliographystyle{myamsalpha}
\bibliography{betti}

\appendix

\section{$K_4$-free graphs}
\label{a:K4}

\paragraph{Preliminaries.}
We keep the notational standards introduced in Section~\ref{sec:prelim}. We
recall, that $N[v] = N_G[v]$ denotes the closed neighborhood of vertex $v$ in a
graph $G$, that is, the set of neighbors of $v$ together with $v$. We, however, modify the definition of the open neighborhood from
Section~\ref{sec:C_4}. Throughout the appendix we assume that $N(v) = N_G(v)$ is
the subgraph of $G$ induced by neighbors of $v$. That is, it is not only the
set of neighbors as in Section~\ref{sec:C_4}. (For further considerations of
the closed neighborhoods, it is not important whether we consider the subgraph
or just the set of vertices.
)

%
%
%

Since we plan to use Lemma~\ref{l:recurrent_simp} quite heavily, it pays off to
set up certain additional notational conventions. Once we fix $v$ and the order of the
neighbors, $v_1, \dots, v_d$ we define $G^i = G - N[v_i] - \{v_1,
\dots, v_{i-1}\}$ for $i \in [d]$. That is, the inequality in
Lemma~\ref{l:recurrent_simp} can be rewritten as
\begin{equation}
\label{e:G^i}
\bb(G) \leq \sum\limits_{i=1}^{d}\bb(G^i).
\end{equation}
We also denote by $k_i$ the size of the set $\blue{N[v_i]} \cup \{v_1, \dots,
v_{i-1}\}$, that is $G^i$ has $n_i = n- k_i$ vertices.



\begin{lemma}
\label{l:cut}
Let $v_1,\dots, v_d$ be vertices forming a cut in $G$ and let $C$ be one of the
components of $G - \{v_1, \dots, v_d\}$ and $G'$ be the union of the remaining
components. Then
$$
\bb(G) \leq \bb(C)\bb(G') + \sum\limits_{i=1}^{d} \bb(G^i).
$$
\end{lemma}

The proof of this lemma is essentially the same as the proof of
Lemma~\ref{l:recurrent_simp} in Adamaszek's paper~\cite{Adamaszek:Betti}.

\begin{proof}
This lemma is obtained by an iterative application of Lemma~\ref{l:single_vertex}. We remove all the vertices $v_1, \dots, v_d$ one by one in the given order.
Finally, we use that $\bb(G - \{v_1, \dots, v_d\}) = \bb(C)\bb(G')$ by
Lemma~\ref{l:du}.
\end{proof}
\paragraph{The main bound.}
We prove the following bound for $K_4$-free graphs.

\begin{theorem}
\label{t:K4}
Let $G$ be a graph with $n$ vertices and
without an induced copy of $K_4$. Then
$$\bb(G) \leq \Theta_2^n = 2^{n/3} \approx 1.2599^n.$$

If, in addition, $G$ contains a vertex of degree at most $3$ which is not in
a component consisting of a single triangle, then
$$\bb(G) \leq (\Theta_2^{-4} + \Theta_2^{-5} + \Theta_2^{-6})\Theta_2^n.$$

\end{theorem}
This first bound is asymptotically optimal as witnessed by the disjoint union of
triangles.


Given that $\Theta_2^{-4} + \Theta_2^{-5} + \Theta_2^{-6} \approx
0.9618$, the improvement from the second bound is very minor. However, it will
be our crucial tool for ruling out $4$-regular graphs.

\paragraph{Minimal counterexample approach.}
The proof is in principle given by induction in the spirit of previous proofs; however, some
new ingredients are needed. From practical point of view, it is better to
reformulate the induction in this case as the minimal counterexample approach. That
is, we will assume that $G$ is a counterexample to Theorem~\ref{t:K4} with the
least number of vertices and we will gradually narrow the set of possible
counterexamples until we show that such $G$ cannot exist. It is easy to check
that the theorem is valid for $n = 1$ or $n=2$.

\subsection{Roots of suitable polynomials}
As our approach in previous sections suggest, we will need to know the roots of
several suitable polynomials. Here we extend the
considerations from Section~\ref{sec:prelim}.
Given an ordered $t$-tuple of positive integers
$(a_1,\dots, a_t)$, we will consider the equation
\begin{equation}
\label{e:roots}
1 = x^{-a_1} + x^{-a_2} + \cdots + x^{-a_t}.
\end{equation}
This can be understood as a polynomial equation after multiplying with a
suitable power of $x$. We are interested in a solution of this equation for $x
\in [1, \infty)$. Note that the right hand side is at least $1$ for $x=1$ and
  it is a decreasing function in $x$ tending to $0$. Therefore, there is a
  unique solution, which we denote by $r_{a_1,\dots,a_t}$. In our previous
  terminology, $\Gamma_d = r_{d+1,\dots,2d}$ and $\Theta_d = r_{d+1, \dots,
  d+1}$ where there are $d$ arguments. We will frequently use the following simple
  observation.
\begin{lemma}
\label{l:roots}
  Whenever $\Omega$ is a real number such that $\Omega
  \geq r_{a_1,\dots, a_t}$, then $\Omega^n \geq \Omega^{n-a_1} + \cdots +
  \Omega^{n-a_t}$, for any positive integer $n$.
\end{lemma}
\begin{proof}
It is sufficient to prove $1 \geq \Omega^{-a_1} + \cdots + \Omega^{-a_t}$. This
immediately follows from the definition of $r_{a_1,\dots,a_t}$.
\end{proof}


We will need to know the approximative numerical values of $r_{a_1,\dots,a_t}$
for various $t$-tuples $(a_1, \dots, a_t)$, so that we can mutually compare
them. We present the values important for this section in Table~\ref{tab:roots};
we also include some of the important values that we met previously.

We will also often use monotonicity, that is, if $(b_1, \dots, b_t) \geq
(a_1, \dots, a_t)$ entry-by-entry, then $r_{b_1,\dots,b_t} \leq r_{a_1,\dots,a_t}$. This allows us to
skip computing precise values for many sequences $(a_1,\dots, a_t)$.

\begin{table}
\begin{center}
\begin{tabular}{llll}
   & & approx. value & approx. value \\
  $(a_1, \dots, a_t)$ & the root & of the root & of
  $\Theta_2^{-a_1} + \cdots + \Theta_2^{-a_t}$\\
\toprule
$(3,3)$ & $\Theta_2 = 2^{1/3}$ & 1.2599 & 1\\
$(6,6,6,6)$ & $\Theta_2 = 2^{1/3}$ & 1.2599 & 1\\
$(5,7,10,10,11,11,12,12)$ & $r_{5,7,10,\dots,12}$ & 1.2590 &  \\
$(6,6,9,10,11,11,12,13)$ & $r_{6,6,9,\dots,13}$ & 1.2590 &  \\
$(6,6,7,8,9)$ & $r_{6,6,7,8,9}$ & 1.2564 & \\
$(1,7)$ & $r_{1,7}$ & 1.2554 & \\
$(5,6,6,8)$ & $r_{5,6,6,8}$ & 1.2541 & \\
$(5,6,7,7)$ & $r_{5,6,7,7}$ & 1.2519 & \\
$(4,5,6)$ & $\Gamma_3$ & 1.24985 & 0.9618\\
$(5,5,5)$ & $3^{1/5}$ & 1.2457 & 0.9449 \\
$(3,4)$ & $\Gamma_2$ & 1.2207 & 0.8969\\
\end{tabular}
\caption{Solutions of Equation~\eqref{e:roots} for suitable $t$-tuples and
values $\Theta_2^{-a_1} + \cdots + \Theta_2^{-a_t}$ for some of them.}
\label{tab:roots}
\end{center}
\end{table}

\subsection{Initial observations about the minimal counterexample}

\begin{lemma}
\label{l:connected}
Let $G$ be a disconnected graph.
Then $G$ is not a minimal counterexample to Theorem~\ref{t:K4}.
\end{lemma}

\begin{proof}
The proof follows directly from Lemma~\ref{l:du}. Indeed, let $H_1, \dots, H_m$
be the components of $G$, where $m \geq 2$. Let $n_i$ be the size of $H_i$.
For contradiction, let us assume that $G$ is a minimal counterexample to
Theorem~\ref{t:K4}. Then $\bb(H_i) \leq \Theta_2^{n_i}$. If in addition, $H_i$ is
not a triangle and it contains a vertex of degree at most $3$, then $\bb(H_i)
\leq (\Theta_2^{-4} + \Theta_2^{-5} + \Theta_2^{-6})\Theta_2^{n_i}$. Therefore,
Lemma~\ref{l:du} gives $\bb(G) \leq \Theta_2^{n_1 + \dots + n_m}$ and, in
addition, $\bb(G) \leq (\Theta_2^{-4} + \Theta_2^{-5} +
\Theta_2^{-6})\Theta_2^{n_1 + \dots + n_m}$ if at least one $H_i$ is
not a triangle and it contains a vertex of degree at most $3$. This contradicts
that $G$ is a counterexample to Theorem~\ref{t:K4}.
\end{proof}


\subsection{Vertices of degree at most $2$.}

We begin by excluding vertices of degree at most $2$.
\begin{lemma}
\label{l:min3}
Let $G$ be a minimal counterexample to Theorem~\ref{t:K4}.
Then the minimum degree of $G$ is at least $3$.
\end{lemma}

\begin{proof}
For contradiction assume the minimum degree of $G$ is less than $3$.

Trivially, the minimum degree of $G$ cannot be zero (otherwise
$\bb(G) = 0$).

If the minimum degree of $G$ equals $1$, let $v$ be a vertex in $G$ of
degree~$1$. Let $v_1$ be the neighbor of $v$. Lemma~\ref{l:recurrent_simp}
gives
$$\bb(G) \leq \bb(G - N[v_1]).$$
This immediately gives that $G-N[v_1]$ is a smaller counterexample.

It remains to consider the case when the minimum degree of $G$ equals $2$.
Let $v$ be a vertex of degree $2$ and let $v_1$ and $v_2$ be its neighbors. If
possible, we pick $v$ so that $v$, $v_1$ and $v_2$ do not induce a component
consisting of a single triangle.
Lemma~\ref{l:recurrent_simp} gives
$$\bb(G) \leq \bb(G^1) + \bb(G^2).$$
Note that the size of $G^1$, as well as of $G^2$, is
at most $n-3$.

If $G$ is a minimal counterexample to Theorem~\ref{t:K4}, then
$$\bb(G) \leq \Theta_2^{n-3} + \Theta_2^{n-3} = \Theta_2^n$$ which gives the
required contradiction for the first bound in Theorem~\ref{t:K4}.

If, in addition, $v$, $v_1$ and $v_2$ do not induce a component
consisting of a single triangle, then
the size of $G^1$ or of $G^2$ is at most $n-4$.

Since $G$ is a minimal counterexample to Theorem~\ref{t:K4},
we get

$$
\bb(G) \leq (\Theta_2^{n-3} + \Theta_2^{n-4})  =
(\Theta_2^{-3} + \Theta_2^{-4})\Theta_2^n \leq (\Theta_2^{-4} + \Theta_2^{-5} +
\Theta_2^{-6})\Theta_2^n
$$
where the last inequality $\Theta_2^{-3} +
\Theta_2^{-4} < \Theta_2^{-4} + \Theta_2^{-5} + \Theta_2^{-6}$ can be checked
in Table~\ref{tab:roots}. Therefore $G$ is not a counterexample to
Theorem~\ref{t:K4}.
\end{proof}

\subsection{Vertices of degree 3}
We continue our analysis by excluding vertices of degree $3$.

\begin{proposition}
\label{p:min4}
  Let $G$ be a minimal counterexample to Theorem~\ref{t:K4}.
Then the minimum degree of $G$ is at least $4$.
\end{proposition}

We need a number of lemmas ruling out various cases.

\begin{lemma}
\label{l:three_points}
  Let $G$ be a minimal counterexample to Theorem~\ref{t:K4}.
  Then $G$ does not contain a vertex $v$ of degree
  $3$ such that the open neighborhood $N(v)$ consists of three isolated points.
\end{lemma}

\begin{proof}
For contradiction, let us assume that $G$ contains such a vertex $v$ and let
$v_1$, $v_2$, and $v_3$ be its neighbors.
As usual, Lemma~\ref{l:recurrent_simp} gives
$$\bb(G) \leq \bb(G^1) + \bb(G^2) + \bb(G^3).$$
We already know that the minimum degree of $G$ is at least $3$
by Lemma~\ref{l:min3}. Since $v_1$, $v_2$, and $v_3$ are three isolated points
we get that the sizes of the three graphs on the right-hand side are at least
$n-4$, $n-5$ and $n-6$.

If $G$ is a minimal counterexample to Theorem~\ref{t:K4}, by Lemma~\ref{l:roots}
we get
$$\bb(G) \leq \Theta_2^{n-4} + \Theta_2^{n-5} + \Theta_2^{n-6} = (\Theta_2^{-4}
+ \Theta_2^{-5} + \Theta_2^{-6})\Theta_2^n.$$
This is the required contradiction. (Note that we have assumed only the weaker
bound in Theorem~\ref{t:K4} for the graphs $G^1$, $G^2$, and
$G^3$, but we still could derive the stronger bound for $G$.)
\end{proof}

By the previous lemma, we have ruled out a case when a vertex of degree three
sees three isolated vertices. Now we will focus on the case when it sees an
isolated vertex and an edge. At first we do not rule it out
completely but set up some necessary condition.

\begin{lemma}
\label{l:3edge_nec}
  Let $G$ be a minimal counterexample to Theorem~\ref{t:K4}.
  If $G$ contains a vertex $v$ of degree
  $3$ such that the open neighborhood $N(v)$ consists of an edge and an isolated
  vertex, then all neighbors of $v$ have degree $3$.
\end{lemma}

\begin{proof}
We know that the minimum degree of $G$ is at least $3$ by
Lemma~\ref{l:min3}. For contradiction, let us assume that $G$ contains a vertex
$v$ such that it has three neighbors $v_1$, $v_2$ and $v_3$; $\deg v_1 \geq 4$,
$\deg v_2, v_3 \geq 3$, and the induced subgraph of $G$ on $\{v_1, v_2, v_3\}$
consists of an edge and an isolated vertex. Without loss of generality we
assume that $v_1$ and $v_2$ are not connected with an edge (otherwise we swap
$v_2$ and $v_3$).
As usual, Lemma~\ref{l:recurrent_simp} gives
$$\bb(G) \leq \bb(G^1) + \bb(G^2) + \bb(G^3).$$
The size of all three graphs on the right-side is at most $n-5$.


Therefore
$$
\bb(G) \leq 3\Theta_2^{n-5} \leq
(\Theta_2^{-4}+\Theta_2^{-5}+\Theta_2^{-6})\Theta_2^n
$$
since $3\Theta_2^{-5} < \Theta_2^{-4} + \Theta_2^{-5} + \Theta_2^{-6}$ which
follows from Table~\ref{tab:roots} (or from the convexity of the function
$\Theta_2^{x}$).

\end{proof}

Now we may rule out the case of a vertex of degree $3$ which sees two edges in
its neighborhood.

\begin{lemma}
\label{l:P3}
  Let $G$ be a minimal counterexample to Theorem~\ref{t:K4}.
  Then $G$ does not contain a vertex $v$ of degree
  $3$ such that the open neighborhood $N(v)$ consists of the path of length $2$.
\end{lemma}

\begin{proof}
For contradiction, let $v$ be a vertex in $G$ contradicting the statement of
the lemma and let $v_1$, $v_2$ and $v_3$ be its neighbors. Without loss of
generality, $v_3$ is adjacent to $v_1$ and to $v_2$ but $\{v_1, v_2\}$ is not
an edge.

We need to distinguish some cases and subcases.
\begin{enumerate}[(i)]
  \item First we assume that $\deg v_3 = 3$.
    \begin{enumerate}[(a)]
      \item Now we consider a subcase $\deg v_2 = 3$. See
	Figure~\ref{f:deg3_ia}.

      In this subcase let $w_2$ be the unique neighbor of $v_2$ different from
	$v$ and $v_3$. \blue{Then the set $\{v_1,v_2\}$ forms a cut.} Let $C$
	be the edge $vv_3$ and $G' = G - \{v_1,v_2,v_3,v\}$.
	Then Lemma~\ref{l:cut} gives
      $$
      \bb(G) \leq \bb(C)\bb(G') + \bb(G-N[v_1]) + \bb(G - N[v_2] - v_1).
      $$
      We observe that in the graph $G-N[v_1]$, the vertex $v_2$ has degree $1$.
      Thus we further get $\bb(G-N[v_1]) \leq \bb(G-N[v_1] - N[w_2])$ by
      Lemma~\ref{l:recurrent_simp}.

      Note that $\bb(C) = 1$ and the size of $G'$ is $n-4$. We also know that
	the size of $G - N[v_2] - v_1$ \blue{is $n-5$}. Finally, the size of $G-N[v_1] - N[w_2]$ is
      at most $n-6$, even if $w_2$ and $v_1$ are neighbors. As usual, if $G$ is a
      minimal counterexample to Theorem~\ref{t:K4}, we get
      $$
      \bb(G) \leq 1\cdot\Theta_2^{n-4} + \Theta_2^{n-6} + \Theta_2^{n-5}
      $$
      as required.
\begin{figure}
\begin{center}
\includegraphics{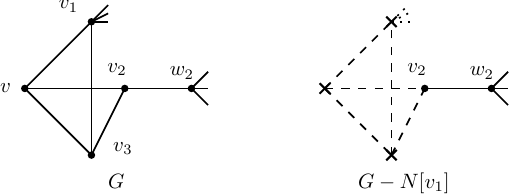}
\caption{Subcase (ia), $G$ and $G - N[v_1]$.}
\label{f:deg3_ia}
\end{center}
\end{figure}
      \item If we consider a subcase $\deg v_1 = 3$, it can be solved
	analogously to the previous subcase by swapping $v_1$ and $v_2$.
      \item Finally, we consider the subcase $\deg v_1 \geq 4$ and $\deg v_2
	\geq 4$. Here we use the usual bound via Lemma~\ref{l:recurrent_simp}
	which gives
      $$\bb(G) \leq \bb(G^1) + \bb(G^2) + \bb(G^3).$$
      The size of $G^1$ is at most $n-5$, the
      size of $G^2$ is at most $n-6$, and the size of $G^3$ is at most $n-4$. Therefore, we get a contradiction as above.
%
  \end{enumerate}
  \item Now we consider the case $\deg v_3 \geq 4$.

    If at least one of the vertices $v_1$ or $v_2$ has degree $4$, or if $\deg
    v_3 \geq 5$, we use again the bound
     $$\bb(G) \leq \bb(G^1) + \bb(G^2) + \bb(G^3).$$
    The sizes of the three graphs on the right-hand side are either at least $n-5$
    or they are at least $n-4$, $n-5$ and $n-6$ respectively. This yields the
    required contradiction eventually using that $3\Theta_2^{-5} <
    \Theta_2^{-4} + \Theta_2^{-5} + \Theta_2^{-6}$.

    Finally, we know that $\deg v_1 = \deg v_2 = 3$ and $\deg v_3 = 4$. In such
    case either $v_3$ and $v_1$ have a single common neighbor (namely $v$), or $v_3$ and
    $v_2$ have a single common neighbor (again $v$), or we get the graph on
    Figure~\ref{f:deg3_ii}.
    (Indeed, if the rightmost vertex in Figure~\ref{f:deg3_ii} has degree at least $4$ then repeating the analysis above in (ii) for $v_1$ instead of $v$ gives the desired contradiction.)
    In the first case, we get a contradiction with
    Lemma~\ref{l:3edge_nec} for $v_1$. The second case is symmetric. In the
    last case, the independence complex of $G$ consists of two edges and
    an isolated vertex; therefore $\bb(G) = 2 \leq (\Theta_2^{-4} +
    \Theta_2^{-5} + \Theta_2^{-6})\Theta_2^5$. A
    contradiction.

\begin{figure}
\begin{center}
\includegraphics{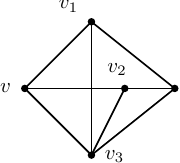}
\caption{A graph occurring in case (ii).}
\label{f:deg3_ii}
\end{center}
\end{figure}
\end{enumerate}
\end{proof}

Now we may rule out the only remaining case of minimum degree $3$ when we have
$3$-regular graph where the open neighborhood $N(v)$ of every vertex $v$
consists of an edge and an isolated vertex.

\begin{lemma}
\label{l:P2+I1}
  Let $G$ be a minimal counterexample to Theorem~\ref{t:K4}.
  Then $G$ is not a cubic ($3$-regular) graph such that the open neighborhood $N(v)$ of
  every vertex $v$ consists of an edge and isolated vertex.
\end{lemma}

\begin{proof}
  For contradiction assume $G$ is a minimal counterexample to Theorem~\ref{t:K4} and
  $G$ satisfies the condition (C) that the open neighborhood $N(v)$ of
    every vertex $v$ consists of an edge and isolated vertex. Equivalently, the
    condition (C) can be reformulated so that $G$ is a cubic graph where every
    vertex is incident to exactly one triangle.
%
    By contracting each triangle to a point,
    graphs satisfying (C)
    are in one to one correspondence with $3$-regular multigraphs. (We allow
    multiple edges but we disallow loops.) See some examples on
    Figure~\ref{f:contract_triangles}. Let $G'$ be the multigraph obtained from
    $G$ by contracting the triangles of $G$.
\begin{figure}
\begin{center}
\includegraphics{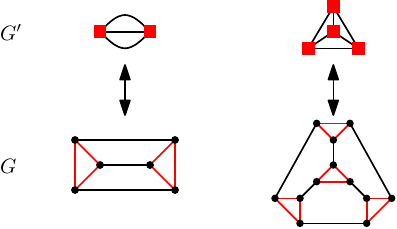}
\caption{One to one correspondence between $G$ and $G'$.}
\label{f:contract_triangles}
\end{center}
\end{figure}
 First, we show that $G'$ is actually a graph. Indeed, if $G'$ contains a triple
 edge, then $G'$ must be the graph on the left part of
 Figure~\ref{f:contract_triangles}, as $G$ is connected. In such case $\bb(G) = 1$ since the
 independence complex of $G$ is the 6-cycle. If $G'$ contains a double edge,
 then $G$ contains a subgraph as on Figure~\ref{f:double_edge}. The vertices $a$
 and $v_1$ may or may not be neighbors. Let $C$ be the subgraph of $G$ induced
 by the vertices $u, w, x, y, z$. We get $\bb(C) = 0$ since the independence
 complex of $C$ is the path of length $4$. Consequently, Lemma~\ref{l:cut}
 gives
 $$
 \bb(G) \leq \bb(G - N[v_1]) + \bb(G - N[v_2] - v_1) \leq \Theta_2^{n-4} +
 \Theta_2^{n-5} < (\Theta_2^{-4} + \Theta_2^{-5} + \Theta_2^{-6})\Theta_2^n.
 $$
 This yields the required contradiction.

 \begin{figure}
\begin{center}
\includegraphics{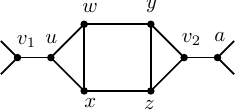}
\caption{Part of $G$ corresponding to a double-edge.}
\label{f:double_edge}
\end{center}
\end{figure}

Now we know that $G'$ is a graph. We distinguish two cases:  either $G'$
contains an induced path of length $2$ or not.

If $G'$ does not contain an induced path of length $2$, let us consider
any vertex $s$ of $G'$. We get that any pair of neighbors of $s$ is adjacent.
Therefore $G'$ is the graph $K_4$ (as $G'$ is cubic and connected as well). We get that $G$ is the graph on the right
part of Figure~\ref{f:contract_triangles} and it remains to bound $\bb(G)$ for this
particular graph. We need to show that $\bb(G) \leq \Theta_2^{12}(\Theta_2^{-4}
+ \Theta_2^{-5} + \Theta_2^{-6}) \approx 15.3893$. Given that $b(G)$ is an
integer, our aim is in fact to show $\bb(G) \leq 15$.

We choose a vertex $v$ of $G$ arbitrarily, and we choose its neighbors $v_1,
v_2, v_3$ so that $v_2$ and $v_3$ are adjacent. We use the usual bound
via Lemma~\ref{l:recurrent_simp} which gives
      $$\bb(G) \leq \bb(G^1) + \bb(G^2) + \bb(G^3).$$
This bound would not be in general sufficient for a vertex $v$ with such a
neighborhood; however, we will show that three summands on the right
hand-side are small enough integers for the graph at hand. The size of $G^1$ is
$8$, the sizes of $G^2$ and $G^3$ are $7$. Since $G$ is a minimal
counterexample, the first summand may be bounded by $\Theta_2^8 \approx
6.3496$. Given that this is an integer, we bound the first summand by $6$.
Similarly, we can bound the remaining two summands, but this time we use the
stronger conclusion of Theorem~\ref{t:K4} which allows to bound each of the
summands by $\Theta_2^7(\Theta_2^{-4} + \Theta_2^{-5} + \Theta_2^{-6}) \approx
4.8473$, that is, we may bound these summands by $4$. Altogether, we get
$\bb(G) \leq 14$ which gives the required contradiction.

Finally, it remains to consider the case when $G'$ contains an induced path of
length $2$. In this case, $G$ contains the subgraph on
Figure~\ref{f:three_triangles}. (Some of the pairs of vertices $w_i$ and $w_j$
may be adjacent.)

\begin{figure}
\begin{center}
\includegraphics{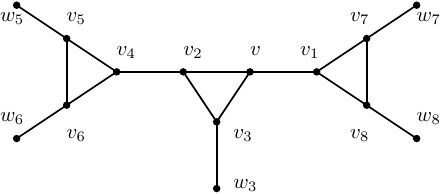}
\caption{Part of $G$ corresponding to an induced path of length $2$.}
\label{f:three_triangles}
\end{center}
\end{figure}

We use the usual bound
via Lemma~\ref{l:recurrent_simp} which gives
      $$\bb(G) \leq \bb(G^1) + \bb(G^2) + \bb(G^3).$$

For the required contradiction, it would not be sufficient to check the orders of the graphs on the right
hand-side. However, we may get a better bound for $\bb(G^1)$
by applying Lemma~\ref{l:recurrent_simp} again to this graph.

For the cut $\{v_4, v_3\}$ in $G^1$ we get
$$
\bb(G - N[v_1]) \leq \bb(G - N[v_1] - N[v_4]) + \bb(G - N[v_1] - N[v_3] - v_4).
$$
The orders of the two graphs on the right hand-side of this inequality are
$n-8$. The orders of the graphs $G- N[v_2] - v_1$ and $G-N[v_3] -
      \{v_1,v_2\}$ are $n-5$. Since $G$ is a minimal counterexample, we get
$$
\bb(G) \leq \Theta_2^n(2\Theta_2^{-5} + 2\Theta_2^{-8}) =
\Theta_2^n(3\Theta_2^{-5}) \leq \Theta_2^n(\Theta_2^{-4} + \Theta_2^{-5} +
\Theta_2^{-6})
$$
as required. (The equality in the middle follows since $\Theta_2 = 2^{1/3}$;
the last inequality follows from Table~\ref{tab:roots}.) This gives the
required contradiction.
\end{proof}

Now we conclude everything to a proof of Proposition~\ref{p:min4}.

\begin{proof}[Proof of Proposition~\ref{p:min4}]
  Let $G$ be a minimal counterexample to Theorem~\ref{t:K4}. By
  Lemma~\ref{l:min3} we know that the minimum degree of $G$ is at least $3$. It
  is, therefore, sufficient to show that $G$ does not contain a vertex of
  degree $3$.

  Since $G$ is $K_4$-free the open neighborhood of any vertex must be
  triangle-free. That is, the open neighborhood of any vertex are either three
  isolated points; an edge and a point; or a path of length $2$. Any of these
  options is ruled out by Lemmas~\ref{l:three_points};~\ref{l:P3}; and~\ref{l:3edge_nec}
  and~\ref{l:P2+I1}, respectively.
\end{proof}

\subsection{Vertices of degree at least $6$}


Now we bound the maximum degree of a possible minimal counterexample.

\begin{lemma}
\label{l:max5}
Let $G$ a minimal counterexample to Theorem~\ref{t:K4},
then the degree of every vertex of $G$ is at most $5$.
\end{lemma}

\begin{proof}
For contradiction, let $v$ be a vertex
of degree $d \geq 6$ in $G$.
 By Lemma~\ref{l:single_vertex}, we have
$$
\bb(G) \leq \bb(G - v) + \bb(G - N[v]).
$$

Since $G$ is a minimal counterexample to Theorem~\ref{t:K4}, we get that
the right hand side of the inequality above is at most $\Theta^{n-1}_2 + \Theta^{n-(d+1)}_2
\leq \Theta^{n-1}_2 + \Theta^{n-7}_2$.
Since $r_{1,7} < \Theta_2$ (see Table~\ref{tab:roots}), we get $\bb(G) \leq \Theta^{n-1}_2
+ \Theta^{n-7}_2 \leq \Theta^n_2$. Together with Proposition~\ref{p:min4}, this contradicts
the fact that $G$ is a counterexample to Theorem~\ref{t:K4}.

\end{proof}

\subsection{Vertices of degree $4$}

We continue our analysis by excluding vertices of degree $4$. As above, we let
$G$ to be a minimal counterexample on $n$ vertices.
By Proposition~\ref{p:min4} we know that the minimum degree of
$G$ is at least $4$, and by Lemma~\ref{l:max5}, the maximum degree of $G$ is at
most $5$. These are already quite restrictive conditions. On the other hand,
the treatment of vertices of degree $4$ is perhaps the most complicated part of
the proof of Theorem~\ref{t:K4}.

We consider a vertex $v$ of degree $4$ (if it exists). We check its open neighborhood $N(v)$,
and depending on $N(v)$ and on the degrees of vertices of $N(v)$ in $G$, we rule out many
cases how may $N(v)$ look like. Once we rule out these cases, we get graphs
with certain structure; and this structure helps us to estimate $\bb(G)$ more
precisely. This will rule out the remaining cases.

Let $v_1,
\dots, v_4$ denote the vertices of $N(v)$, and recall the bound~\eqref{e:G^i}
$$
\bb(G) \leq \bb(G^1) + \bb(G^2) + \bb(G^3) + \bb(G^4),
$$
where $G^i$ stands for $G - N[v_i] - \{v_1, \dots, v_{i-1}\}$.
We will often alternate the order of the vertices $v_1, \dots, v_4$ in order to
get the best bound.

We also recall that $k_i$ is set up so that
$G^i$ has $n- k_i$ vertices. If we show that
$r_{k_1,\dots,k_4} \leq \Theta_2$, then we are done, since we obtain
$$
\bb(G) \leq \Theta_2^n
$$
by Lemma~\ref{l:roots}.
This is the required contradiction.
In particular, we achieve this task, if $(k_1, \dots, k_4) \geq (6,6,6,6)$ or
$(k_1, \dots, k_4) \geq (5,6,7,7)$, up to possibly permuting $k_1, \dots, k_4$;
see Table~\ref{tab:roots}.
(This is not the same as permuting $v_1, \dots, v_4$; permuting the vertices
may yield an essentially different values of $k_1, \dots, k_4$.) On the other
hand, it is insufficient to achieve that $(k_1, \dots, k_4)$ is $(5,5,7,7)$
or $(5, 6, 6, 7)$, since $r_{k_1,\dots,k_4} > \Theta_2$ in these cases (very
tightly). This will complicate our analysis.

Now, let us inspect the possible neighborhoods $N(v)$. Since $G$ is $K_4$-free,
we get that $N(v)$ is triangle-free. There are $7$ options for the isomorphism
class of $N(v)$ depicted on Figure~\ref{f:neighborhoods4}.

\begin{figure}
\begin{center}
\includegraphics{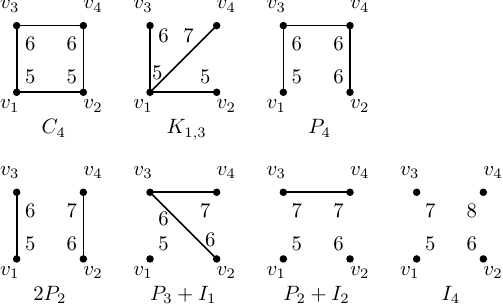}
\caption{The possible isomorphism classes of $N(v)$.
The number at a vertex $v_i$ denotes $k_i$ under the condition that the degree of
$v_i$ in $G$ is $4$.
}
\label{f:neighborhoods4}
\end{center}
\end{figure}

The discussion above immediately gives that the last two options cannot occur for a minimal
counterexample.

\begin{lemma}
\label{l:I4}
  Let $G$ be a minimal counterexample to Theorem~\ref{t:K4}. Then $G$ does not
  contain vertex $v$ such that $N(v)$ is isomorphic to $I_4$ or $P_2 + I_2$; see
  Figure~\ref{f:neighborhoods4}.
\end{lemma}

\begin{proof}
For contradiction, there is such $v$ in the minimal counterexample.
  We already know that we may assume that the minimum degree of $G$ is at least
$4$. Therefore, from the discussion above we get $(k_1, \dots, k_4) \geq
(5,6,7,7)$, which yields the required contradiction.
\end{proof}

Our next task is to show that if $G$ is a minimal counterexample which contains
a vertex of degree $4$, then $G$ is actually $4$-regular. We do this in two
steps. First we significantly restrict the possible isomorphism classes of
$N(v)$ where $v$ is a vertex of degree $4$ incident to a vertex of degree $5$.
Next, we analyze the remaining options in more details so that we may rule them
out as well.

\begin{lemma}
\label{l:45}
  Let $G$ be a minimal counterexample to Theorem~\ref{t:K4}. Let $v$ be a
  vertex of degree $4$ in $G$, which is incident to a vertex of degree greater or equal to
  $5$. Then one of the following options hold.
  \begin{enumerate}[$(a)$]
    \item $N(v)$ is isomorphic to $C_4$, one vertex of $N(v)$ has degree $5$ in
      $G$ and the three remaining vertices have degrees $4$ in $G$.
    \item $N(v)$ is isomorphic to $C_4$, two opposite vertices of $N(v)$ have
      degrees $5$ in
      $G$ and the two remaining vertices have degrees $4$ in $G$.
    \item $N(v)$ is isomorphic to $K_{1,3}$, one vertex of $N(v)$ has degree $5$ in
      $G$ and the three remaining vertices have degrees $4$ in $G$.
    \end{enumerate}
\end{lemma}

\begin{proof}
Let $v$ be the vertex from the statement. We gradually exclude all remaining
cases. By Lemma~\ref{l:I4}, we already know that $N(v)$ is not isomorphic to $P_2
+ I_2$ or $I_4$.

First let us consider the case that $N(v)$ is isomorphic to $2P_2$ or $P_3 +
I_1$. Let us choose $v_1, \dots, v_4$ according to
Figure~\ref{f:neighborhoods4}. Since one of the vertices $v_1, \dots, v_4$ has
degree $5$, we get that $(k_1, \dots, k_4) \geq (6,6,6,7)$ or $(k_1, \dots,
k_4) \geq (5,6,7,7)$ or $(k_1, \dots, k_4) \geq (5,6,6,8)$. Therefore this
option is excluded since the three roots $r_{6,6,6,6}$, $r_{5,6,7,7}$, and
$r_{5,6,6,8}$ are less than $\Theta_2$; see Table~\ref{tab:roots}.

Now let us consider the case that $N(v)$ is isomorphic to $P_4$. At least one
vertex of $N(v)$ has degree $5$ in $G$. Up to isomorphism, there are two
(non-exclusive) options depicted at Figure~\ref{f:P4}. Depending on these
options we label the vertices of $N(v)$ by $v_1, \dots, v_4$ according to
Figure~\ref{f:P4}. In both cases, we get $(k_1, \dots, k_4) \geq (6,6,6,6)$
which contradicts that $G$ is a minimal counterexample.

\begin{figure}
\begin{center}
\includegraphics{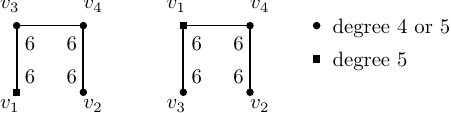}
\caption{Excluding the case that $N(v)$ is isomorphic to $P_4$. The label at
$v_i$ denotes $k_i$ under the condition that all bullet vertices have degree
$4$. (In general, it is a lower bound for $k_i$.)}
\label{f:P4}
\end{center}
\end{figure}

Let us continue with the case that $N(v)$ is isomorphic to $K_{1,3}$.
If there is only one vertex in $N(v)$ of degree $5$ in $G$, we get the
case $(c)$ of the statement of this lemma. Therefore, we may assume that
$N(v)$ contains at least two vertices of degree $5$ in $G$ and we want to
exclude this case. We label the vertices of $N(v)$ by $v_1, \dots, v_4$
according to Figure~\ref{f:neighborhoods4}. Up to a self-isomorphism of
$K_{1,3}$ we may assume that $v_2$ has degree $5$ and also $v_1$ or
$v_3$ has degree $5$. Therefore, we get $(k_1, \dots, k_4) \geq (6,6,6,7)$ or
$(k_1, \dots, k_4) \geq (5,6,7,7)$ which gives the required contradiction.

Finally, it remains to consider the case that $N(v)$ is isomorphic to $C_4$. In
this case, it is sufficient to exclude the case that $N(v)$ contains two vertices of
degree $5$ in $G$ which are neighbors. If there are two such vertices, we label
the vertices of $N(v)$ by $v_1, \dots, v_4$ so that $v_1$ and $v_2$ have
degrees $5$. Then $(k_1, \dots, k_4) \geq (6, \dots, 6)$ which is the required
contradiction.
\end{proof}

\begin{figure}
\begin{center}
\includegraphics{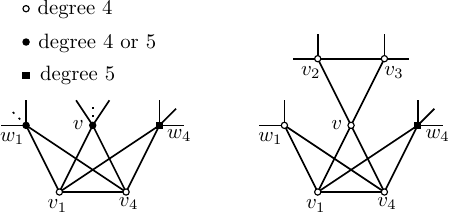}
\caption{Forbidden subgraph.}
\label{f:three_peaks}
\end{center}
\end{figure}

\begin{lemma}
\label{l:three_peaks}
  Let $G$ be a minimal counterexample to Theorem~\ref{t:K4}. Then $G$ does not
  contain the graph on $5$ vertices from Figure~\ref{f:three_peaks}, left, as an
  induced subgraph, where $\deg(v_1) = \deg(v_4) = 4$, $\deg(w_4) = 5$ and
  $\deg(w_1), \deg(v) \in \{4,5\}$.
\end{lemma}

\begin{proof}
  Assume by contradiction $G$ contains such a subgraph, and observe that the neighborhood $N(v_1)$ is
  isomorphic to $K_{1,3}$. Therefore, Lemma~\ref{l:45} gives that $\deg(v) =
  \deg(w_1) = 4$.

  Now, let us focus on $N(v)$. Since we know all neighbors of $v_1$ and $v_4$,
  we get that $N(v)$ is isomorphic to one of the graphs $2P_2$ or $P_2 + I_2$.
  But Lemma~\ref{l:I4} excludes the latter case. In addition, Lemma~\ref{l:45}
  implies that all neighbors of $v$ have degree $4$ in $G$. Let $v_2$ and $v_3$
  be the two remaining neighbors of $v$; see Figure~\ref{f:three_peaks}, right.

  Now, we want to use Lemma~\ref{l:recurrent_simp} on $v$. In this case $(k_1,
  \dots, k_4) = (5,6,6,7)$ which is not sufficient but we may gain a slight
  improvement if we inspect the graphs on the right hand-side of
  Lemma~\ref{l:recurrent_simp} in this case:
  \begin{equation}
  \bb(G) \leq \bb(G^1) + \bb(G^2) + \bb(G^3) + \bb(G^4).
  \label{e:peaks}
  \end{equation}
  We have that the size of $G^i$ is $n - k_i$. However, we may also check that
  $G^2 = G - N[v_2] - v_1$ contains a vertex of degree at most $3$ which is not
  contained in a component of $G^2$ consisting of a single triangle. Indeed,
  $v_4$ is such a vertex. (Note that $w_1$ or $w_4$ may or may not belong to
  $G^2$). Similarly, $G^3$ contains a vertex of degree at most $3$ which is not
    contained in a component of $G^3$ consisting of a single triangle, which is
    again witnessed by $v_4$.

  Therefore,
  since $G$ is a minimal counterexample to Theorem~\ref{t:K4}, we get
  $$
  \bb(G^2), \bb(G^3) \leq \Theta_2^{n-6}(\Theta_2^{-4} + \Theta_2^{-5} +
  \Theta_2^{-6}).
  $$
  Hence~\eqref{e:peaks} gives
  $$
  \bb(G) \leq \Theta_2^n (\Theta_2^{-5} + (\Theta_2^{-6} +
  \Theta_2^{-6})(\Theta_2^{-4} + \Theta_2^{-5} +
  \Theta_2^{-6}) + \Theta_2^{-7}).
  $$

  We get a contradiction to the assumption that $G$ is a counterexample to Theorem~\ref{t:K4} as $r_{5,7,10,10,11,11,12,12} \leq \Theta_2$; see Table~\ref{tab:roots}.

\end{proof}

Now we have enough tools to exclude the remaining cases of Lemma~\ref{l:45}.

\begin{lemma}
\label{l:4regular}
  Let $G$ be a minimal counterexample to Theorem~\ref{t:K4}. If $G$ contains a
vertex of degree $4$, then $G$ is $4$-regular.
\end{lemma}

\begin{proof}
  We know that $G$ is connected by Lemma~\ref{l:connected}, and has minimal degree $4$ by Proposition~\ref{p:min4}. For contradiction, let us suppose that $G$
contains a vertex of degree $4$ but $G$ is not $4$-regular. In particular, $G$
contains a vertex $v$ of degree $4$ which is incident to a vertex of degree
$5$; thus one of the three options (a,b,c) in Lemma~\ref{l:45} must hold.

First, let us consider the case $(c)$, that is, $N(v)$ is isomorphic to
$K_{1,3}$ and $v$ is incident to exactly one vertex of degree $5$. Let us label
the vertices of $N(v)$ as in Figure~\ref{f:neighborhoods4}. Now, there are two
subcases, either $v_1$ is the vertex of degree $5$, or, without loss of
generality, $v_2$ is the vertex of degree $5$.

In the first subcase, $v_1$ has a single neighbor $w$ different from $v$, $v_2$,
$v_3$ and $v_4$.
Now let us describe $N(v_2)$. By Lemma~\ref{l:45}, $N(v_2)$ is
isomorphic to $C_4$ or $K_{1,3}$ ($v_1$ is a neighbor of $v_2$ of degree $5$).
In addition, $v$ and $v_1$ belong to $V(N(v_2))$. We also have
$\deg_{N(v_2)}(v) = 1$ since $\deg_G(v) = 4$ and $v_2$ and $v_3$ are not
neighbors of $v_2$. Similarly, we deduce $\deg_{N(v_2)}(v_1) \leq 2$. This
rules out both options, $C_4$ and $K_{1,3}$ for the isomorphism class of
$N(v_2)$. A contradiction.

In the second subcase, we suppose that $v_2$ has degree $5$ in $G$. Therefore,
$v_1$ has degree $4$ in $G$ and consequently $N(v_1)$ is isomorphic to
$K_{1,3}$. Therefore, up to relabeling of the vertices, $G$ contains the induced
subgraph from Lemma~\ref{l:three_peaks}. This gives the required contradiction.

This way, we have ruled out the case $(c)$ of Lemma~\ref{l:45}. Therefore, we
may assume that any vertex $v$ of degree $4$ of $G$, incident to a vertex of
degree $5$, falls into the case $(a)$ or $(b)$ of Lemma~\ref{l:45}.

Now, let us consider an arbitrary vertex $w$ of degree $5$, incident to a
vertex $v$ of degree $4$. Our aim is to show that $N(w)$ is isomorphic either
to $C_5$, or to $C_4 + I_1$; see Figure~\ref{f:C5_C4}.

By inspecting $N(v)$, we get that $v$ and $w$ have
two common neighbors, say $v_1$ and $v_2$, which are not incident. Now, we
analogously inspect the $4$-vertex graphs $N(v_1)$, $N(v_2)$ and so on (for the other vertices of degree $4$ incident to $w$), and we arrive at one of the two cases in Figure~\ref{f:C5_C4} (keeping the degree notation of Figure~\ref{f:three_peaks}).

\begin{figure}
\begin{center}
\includegraphics{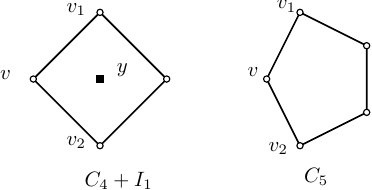}
\caption{Possible isomorphism types of $N(w)$.}
\label{f:C5_C4}
\end{center}
\end{figure}

Therefore, it is sufficient to distinguish two subcases according to the
isomorphism type of $N(w)$.

First we suppose that $N(w)$ is isomorphic to $C_5$. Then
all vertices of $N(w)$ have degree $4$ in $G$. Let us label the vertices
of $N(w)$ according to Figure~\ref{f:C5_C4}
and let $x$ be the neighbor of $v$ different from $w$, $v_1$ and $v_2$. By
checking $N(v)$ again, we see that $v_1$ and $v_2$ are neighbors of $x$ as
well. Then by checking $N(v_1)$ and $N(v_2)$ we get that all vertices of $N(w)$
are incident to $x$, and we get that $G$ is the graph on
Figure~\ref{f:bipyramid}. In this case, we easily observe that the independence
complex of $G$ consists of an edge and a cycle. Therefore $\bb(G) = 2 \leq
\Theta_2^7$ which is the required contradiction.

\begin{figure}
\begin{center}
\includegraphics{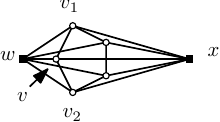}
\caption{The resulting $G$ if $N(w)$ is isomorphic to $C_5$.}
\label{f:bipyramid}
\end{center}
\end{figure}

Now, we suppose that $N(w)$ is isomorphic to $C_4 + I_1$. Let us again label the
vertices of $N(w)$ according to Figure~\ref{f:C5_C4}. In this case, the
four vertices of $C_4$ have degree $4$ in $G$. (The last vertex $y$ has degree
$5$ in $G$, but we do not need this information.) Analogously to the previous
case, we deduce that there is another vertex $x$ incident to the vertices of
$C_4$ in $N(w)$. The degree of $x$ in $G$ may be $4$ or $5$. See
Figure~\ref{f:octahedron}.
\begin{figure}
\begin{center}
\includegraphics{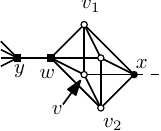}
\caption{The resulting $G$ if $N(w)$ is isomorphic to $C_4 + I_1$.}
\label{f:octahedron}
\end{center}
\end{figure}
%
%
Let us apply Lemma~\ref{l:cut} to the cut formed by the
vertices $w$ and $x$ (in this order, which is relevant for the lemma). We obtain
$$
\bb(G) \leq 1\cdot \Theta_2^{n-6} + \Theta_2^{n-6} + \Theta_2^{n-6}= \frac34\Theta_2^n.
$$
As usual, this contradicts that $G$ is a counterexample to Theorem~\ref{t:K4}.
\end{proof}

\paragraph{4-regular graphs.}
Now we know that if a minimal counterexample $G$ contains a vertex of degree
$4$, then it must be $4$-regular. Our next step is to rule out this case.

We will often need to check that a certain graph satisfies the stronger
condition in the statement of Theorem~\ref{t:K4}. Here is a useful sufficient
condition which allows us to avoid distinguishing various special cases.

\begin{lemma}
\label{l:no_triangle}
Let $G$ be a connected $4$-regular graph and let $H$ be a proper subgraph of
$G$ such that the number of vertices of $H$ is not divisible by $3$. Then $H$
contains a vertex of degree at most $3$ in $H$ which is not in a component
consisting of a single triangle.
\end{lemma}

\begin{proof}
Let us consider a component $C$ of $H$ which has the number of vertices not
divisible by $3$. In particular $C$ is not a triangle.
Since $H$ is a proper subgraph of a connected $4$-regular
graph, $C$ must contain a vertex of degree at most $3$.
\end{proof}

In Lemma~\ref{l:I4} we have ruled out certain options for the neighborhood of a
vertex of degree $4$. Now, we may rule out further options.

\begin{lemma}
\label{l:2P2}
  Let $G$ be a minimal counterexample to Theorem~\ref{t:K4}. Then $G$ does not
  contain vertex $v$ such that $N(v)$ is isomorphic to $2P_2$ or $P_3 + I_1$; see
  Figure~\ref{f:neighborhoods4}.
\end{lemma}

\begin{proof}
  By Lemma~\ref{l:4regular}, we know that $G$ is $4$-regular. For
  contradiction, let us assume that there is a vertex $v$ such that $N(v)$ is isomorphic
  to $2P_2$ or $P_3 + I_1$. Let us label the neighbors of $v$ according to
  Figure~\ref{f:neighborhoods4}. In our usual notation, this gives $(k_1, k_2,
  k_3, k_4) = (5,6,6,7)$. This is insufficient to rule out these cases
  directly, but it will help us to focus on the stronger conclusion of
  Theorem~\ref{t:K4}. Lemma~\ref{l:recurrent_simp} gives
  $$
  \bb(G) \leq \bb(G^1) + \bb(G^2) + \bb(G^3) + \bb(G^4)
  $$
  where $G^i = G - N[v_i] - \{v_1,\dots, v_{i-1}\}$ as usual.
  Now, let us consider two cases depending on whether the number of vertices of
  $G$ is divisible by $3$. If it is divisible by $3$, Lemma~\ref{l:no_triangle},
  together with the fact that $G$ is a minimal counterexample, gives
  $$
  \bb(G) \leq \Theta_2^{n-5}(\Theta_2^{-4} + \Theta_2^{-5} + \Theta_2^{-6}) +
  \Theta_2^{n-6} + \Theta_2^{n-6} + \Theta_2^{n-7}(\Theta_2^{-4} +
  \Theta_2^{-5} + \Theta_2^{-6}).
  $$
  If the number of vertices of $G$ is not divisible by $3$, we analogously get
  $$
  \bb(G) \leq \Theta_2^{n-5} + 2\Theta_2^{n-6}(\Theta_2^{-4} + \Theta_2^{-5} + \Theta_2^{-6}) +
        \Theta_2^{n-7}.
 $$
 In both cases, we get the required contradiction, since
 $r_{6,6,9,10,11,11,12,13} \leq \Theta_2$ as well as $r_{5,7,10,10,11,11,12,12}
 \leq \Theta_2$. See Table~\ref{tab:roots}.
\end{proof}

Now, we may also rule out an open neighborhood isomorphic to $K_{1,3}$.

\begin{lemma}
\label{l:K13}
  Let $G$ be a minimal counterexample to Theorem~\ref{t:K4}. Then $G$ does not
  contain a vertex $v$ such that $N(v)$ is isomorphic to $K_{1,3}$; see
  Figure~\ref{f:neighborhoods4}.
\end{lemma}

\begin{proof}
  For contradiction, there is such a vertex $v$. Let us label the neighbors of
  $v$ according to Figure~\ref{f:neighborhoods4}.
  By Lemma~\ref{l:4regular}, we know that $G$ is $4$-regular. Therefore, the
  only common neighbor of $v_2$ and $v_1$ is $v$. Similarly, the only common
  neighbor of $v_2$ and $v$ is $v_1$. Therefore $N(v_2)$ must be isomorphic to
  $2P_2$ or to $P_2 + I_2$. However this is already ruled out by
  Lemmas~\ref{l:I4}~and~\ref{l:2P2}.
\end{proof}

Now let us establish two graph classes that will help us to work with
$4$-regular graphs such that the open neighborhood of every vertex is isomorphic
either to the cycle $C_4$ or to the path on $4$ vertices $P_4$. The
\emph{triangular path} on $n$ vertices is the graph $TP_n$ such that $V(TP_n) :=
[n]$ and
$$E(TP_n) := \left\{ij \in \binom{[n]}2 \colon |i - j| \leq 2\right\}.$$
Similarly, we
define \emph{triangular cycle} so that we consider the distance cyclically.
That is, we get a graph $TC_n$ such that $V(TC_n) := V(TP_n) = [n]$
and
$$E(TC_n) := \left\{ij \in \binom{[n]}2 \colon i - j \pmod n \in \{n-2,n-1,1,2\}\right\}.$$
See Figure~\ref{f:TC}.

\begin{figure}
\begin{center}
\includegraphics{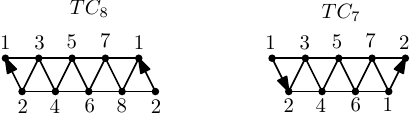}
\caption{The graphs $TC_8$ and $TC_7$ (after the identification of the vertices
labeled $1$ and $2$).}
\label{f:TC}
\end{center}
\end{figure}

If we consider the clique complex $\cl(TC_n)$, then we get a triangulation of
an annulus for $n \geq 8$ even, whereas we get a triangulation of the
M\"{o}bius band for $n \geq 7$ odd. We establish the following structural
result for graphs with the remaining two options for open neighborhoods.

\begin{lemma}
\label{l:C4P4}
Let $G$ be a connected $4$-regular graph such that the open neighborhood of every vertex
is isomorphic either to $C_4$ or to $P_4$.
Then $G$ is isomorphic to $TC_n$ for some $n \geq 6$.
\end{lemma}

\begin{proof}
Let us consider the clique complex $\cl(G)$. By the condition on the
neighborhoods, we get that $\cl(G)$ is a triangulated surface, possibly with
boundary.

Let $k$ be the number of vertices of $G$ such that their neighborhood is
isomorphic to $C_4$ and $\ell$ be the number of remaining vertices. Therefore,
by double counting, we get that $\cl(G)$ has $k + \ell$ vertices, $2(k+\ell)$
edges, and $\frac43k + \ell$ triangles. That is, the Euler characteristic
$\chi(\cl(G))$ equals $k+\ell - 2(k + \ell) + \frac43k + \ell = \frac13k$.
However; surfaces with nonnegative Euler characteristic are rare, which will
help us to rule out many options.

It is easy to check that $G$ has at least $6$ vertices because the closed
neighborhood of a single vertex has already $5$ vertices.

If $\ell = 0$, then $k \geq 6$ and therefore $\chi(\cl(G)) \geq 2$. This leaves
an only option that $\cl(G)$ is a sphere, $\chi(\cl(G)) = 2$ and $k = 6$.
Consequently (by checking how to extend a neighborhood of arbitrary vertex), we
get that $G = K_{2,2,2} = TC_6$.

If $\ell > 0$, then $\cl(G)$ must be a surface with boundary and the only
options are the disc (with Euler characteristic $1$), the annulus and the M\"{o}bius
band (the latter two have Euler characteristic $0$).

In the case of a disc, we get $k = 3$.
We say, that a vertex $v$ of $G$ is a
$C_4$-vertex, if $N(v)$ is isomorphic to $C_4$.
The interior vertices of the disc are precisely the $c_4$-vertices.
By a local check of the
neighborhoods, we see that every $C_4$ vertex is adjacent to at least two
$C_4$-vertices, and therefore, the $C_4$-vertices form a triangle in $G$.
We now count the edges according to the number of $C_4$-vertices they contain. Note that no edge of the disc connects two boundary vertices, as such edge $e$ would separate the disc into two regions, and in the region with no interior vertex there would be a boundary vertex, not in $e$, of degree at most $2$; a contradiction.
Thus, each boundary vertex has exactly two neighbors in the boundary and two in the interior. The number of boundary edges is clearly $l$. To summarize, the total number of edges is $3+l+2l$, but it is also $2(3+l)$, thus $l=3$. This is a contradiction as then the triangle on the $3$ boundary vertices is in $\cl(G)$, eliminating the boundary of the disc.


Finally, it remains to consider the case of the annulus and the M\"{o}bius
band. In this case, $k=3\chi(\cl(G)) = 0$; so all the vertices are on the boundary of $\cl(G)$.
Now a simple local inspection gives that $G$ is isomorphic to $TC_n$ for $n
\geq 7$. (We consider an arbitrary
vertex $v$ and its neighborhood $N(v)$, then we check the neighborhoods of the
vertices of $N(v)$ which locally determines the graph uniquely. We continue
this inspection, until we reach a vertex from `two directions'.)
\end{proof}

Now we need to bound $\bb(TC_n)$ for $n \geq 6$ in order to finish the case of
graphs of minimum degree $4$. First, we provide a bound for $\bb(TP_n)$ which
will be useful for bounding $\bb(TC_n)$.

\begin{lemma}
\label{l:TP}
For $n \neq 3$, we have $\bb(TP_n) \leq 2^{n/4}$. Furthermore $\bb(TP_3) = 2$.
\end{lemma}

\begin{proof}
It is easy to determine the first few initial values by checking the corresponding
independence complexes. We obtain $\bb(TP_0) = 1$, $\bb(TP_1) = 0$, $\bb(TP_2) = 1$, $\bb(TP_3)
= 2$, and $\bb(TP_4) = 2$, where $TP_0$ stands for the empty graph.

Next, we use Lemma~\ref{l:recurrent_simp} to vertex $n$ and its neighbors $n-1$
and $n-2$. We get
\begin{equation}
\label{e:TP}
  \bb(TP_n) \leq \bb(TP_{n-4}) + \bb(TP_{n-5}).
\end{equation}
This further gives $\bb(TP_5) \leq 1$, $\bb(TP_6) \leq 1$, $\bb(TP_7) \leq 3$,
and $\bb(TP_8) \leq 4$. Therefore, $\bb(TP_n) \leq 2^{n/4}$ for $n \in [8]
\setminus \{3\}$.
(Note that $2^{7/4} \approx 3.3636$.) Furthermore, it is trivial to show that
$\bb(TP_n) \leq 2^{n/4}$ for $n \geq 9$ by induction using~\eqref{e:TP}.
\end{proof}

Now we bound $\bb(TC_n)$.

\begin{lemma}
\label{l:TC}
  For $n \geq 9$ we have $\bb(TC_n) \leq 2^{n/4}(2^{-1/2} + 2^{-1/4}) \approx
  1.5480 \cdot 2^{n/4}$.
\end{lemma}

\begin{proof}
First remove the vertex $n$ and then the vertex $n-1$ from $TC_n$.
  Lemma~\ref{l:single_vertex} then gives
\begin{align*}
 \bb(TC_n) &  \leq \bb(TC_n - n) + \bb(TC_n - N[n])\\
 &\leq \bb(TC_n - \{n,n-1\})+\bb(TC_n -n-N[n-1])+\bb(TC_n - N[n])\\
 & \leq \bb(TP_{n-2}) + 2\bb(TP_{n-5}).
  \end{align*}
  Therefore, Lemma~\ref{l:TP} gives
  $$
  \bb(TC_n) \leq 2^{(n-2)/4} + 2\cdot2^{(n-5)/4} = 2^{n/4}(2^{-1/2} +
  2^{-1/4}).
  $$
\end{proof}

Now we may rule out $4$-regular graphs.

\begin{lemma}
\label{l:not_4reg}
Let $G$ be a minimal counterexample to Theorem~\ref{t:K4}. Then $G$ is not a
$4$-regular graph.
\end{lemma}

\begin{proof}
  For contradiction, let us assume that there is such $G$. By
  Lemma~\ref{l:connected} we know that $G$ is connected.
  By Lemmas~\ref{l:I4},~\ref{l:2P2} and~\ref{l:K13} we know that the open
  neighborhood of every vertex in $G$ is isomorphic either to $C_4$ or to
  $P_4$. Lemma~\ref{l:C4P4} implies that $G$ is isomorphic to $TC_n$ for $n
  \geq 6$. Therefore, in order to obtain a contradiction, it is sufficient to
  show that $\bb(TC_n) \leq \Theta_2^n = 2^{n/3}$.

  We treat separately the cases $n \in \{6,7,8\}$. The independence complex of
  $TC_6$ consists of three edges and therefore $\bb(TC_6) = 2$. The
  independence complex of $TC_7$ is the cycle $C_7$ which gives $\bb(TC_7) =
  1$. Finally, the independence complex of $TC_8$ is a connected $3$-regular
  graph (triangle-free) with $8$ vertices, thus with $12$ edges. Therefore $\bb(TC_8) = 5$. In all
  three cases, we easily see that $\bb(TC_n) \leq 2^{n/3}$.

  Now we consider $n \geq 9$. Lemma~\ref{l:TC} gives
  $\bb(TC_n) \leq 2^{n/4}(2^{-1/2} + 2^{-1/4}).$
  Therefore, we need to check the inequality $2^{-1/2} + 2^{-1/4} \leq
  2^{n/12}$. This inequality holds for $n \geq 8$ since $2^{8/12} \approx
  1.5874$ while $2^{-1/2} + 2^{-1/4} \approx 1.5480$.

\end{proof}
Proposition~\ref{p:min4} and
Lemmas~\ref{l:max5},~\ref{l:4regular} and~\ref{l:not_4reg} together imply the following
corollary.

\begin{proposition}
\label{p:5reg}
Let $G$ be a minimal counterexample to Theorem~\ref{t:K4}. Then $G$ is a
$5$-regular graph. \qed
\end{proposition}

\subsection{5-regular graphs}

It remains to rule out $5$-regular graphs. We use an analogous approach as in the
case of $4$-regular graphs. Given a minimal counterexample $G$, which is
$5$-regular by Proposition~\ref{p:5reg}, and a vertex
$v$ of $G$, we consider all possible isomorphism classes of $N(v)$. Those are
triangle free graphs on $5$ vertices. All triangle free graphs on $5$
vertices with at least $4$ edges are depicted on Figure~\ref{f:neighborhoods5}.
All other triangle-free graphs on $5$ vertices are subgraphs of $P_5$ or
$K_{1,4}$. (It is easy to check both claims from the well known list of graphs
on $5$ vertices.)

\begin{figure}
\begin{center}
\includegraphics{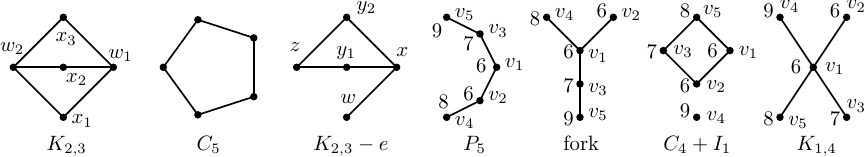}
\caption{Triangle-free graphs with $5$ vertices and at least $4$ edges.}
\label{f:neighborhoods5}
\end{center}
\end{figure}

We use the standard approach via Lemma~\ref{l:recurrent_simp} to rule out the
cases when $N(v)$ does not have many edges.

\begin{lemma}
\label{l:5edges}
Let $G$ be a minimal counterexample to Theorem~\ref{t:K4} and let $v$ be any
vertex of $G$. Then $N(v)$ contains at least $5$ edges.
\end{lemma}

\begin{proof}
  For contradiction, $G$ is a minimal counterexample and $v$ is a vertex of $G$
  such that $N(v)$ contains at most $4$ edges. We know that $G$ is $5$-regular.
  Therefore, $N(v)$ is one of the four graphs with $4$ edges on
  Figure~\ref{f:neighborhoods5}, or their subgraph. Let us label the vertices
  of $N(v)$ according to Figure~\ref{f:neighborhoods5} (we fix one choice of a
  subgraph if $G$ has less than $4$ edges). We use Lemma~\ref{l:recurrent_simp}
  to $v$. In our standard notation, we get $(k_1, \dots, k_5) \geq
  (6,6,7,8,9)$ (if $G = C_4 + I_1$, we have to permute last two coordinates).
  Therefore, $G$ cannot be a counterexample since $r_{6,6,7,8,9} < \Theta_2$;
  see Table~\ref{tab:roots}.
\end{proof}

Therefore, it remains to consider the connected $5$-regular graphs such that the open
neighborhood of every vertex is isomorphic to $K_{2,3}$, $C_5$ or $K_{2,3} -
e$; see Figure~\ref{f:neighborhoods5}. Fortunately, such graphs are very rare.
In fact, we will show that there are only two such graphs. One of them is the
graph of the icosahedron, which we denote by $G_{\ico}$. The second graph is the
join (as a graph, not as a simplicial complex) of $C_5$ and $I_3$ which we denote by
$C_5 \star I_3$. That is, $C_5 \star I_3$ is the graph with the vertex set
$V(C_5 \star I_3) = V(C_5) \cup V(I_3)$, assuming that $V(C_5)$ and $V(I_3)$
are disjoint, and with the set of edges
$$
E(C_5 \star I_3) = E(C_5) \cup \{uv\colon u \in C_5, v \in I_3\}.
$$

\begin{lemma}
\label{l:deg5_cases}
Let $G$ be a $5$-regular graph such that the open
neighborhood of every vertex is isomorphic to $K_{2,3}$, $C_5$ or $K_{2,3} -
e$. Then $G$ is isomorphic to $G_{\ico}$ or to $C_5 \star I_3$.
\end{lemma}

\begin{proof}
  Let us first consider the case that $G$ contains a vertex $v$ such that
  $N(v)$ is isomorphic to $K_{2,3} - e$. Let us label the vertices of $N(v)$
  according to Figure~\ref{f:neighborhoods5}.
  Now, let us focus on $N(w)$. It contains $x$ and $v$, which are neighbors.
  Moreover, $\deg_{N(w)}v = 1$ since $v$ has degree $5$ in $G$ and $v$
  is also incident to $y_1$, $y_2$ and $z$ which are not incident with $w$.
  Therefore $N(w)$ must be isomorphic to $K_{2,3} - e$ as well, which implies
  that $\deg_{N(w)} x = 3$. But this is a contradiction, since $x$ has too many
  neighbors, namely $v, w, y_1, y_2$ and two other neighbors which are incident
  to $w$. Altogether, $G$ cannot contain a vertex such that its open
  neighborhood is isomorphic to $K_{2,3} - e$.

  Now, let us consider the case that $G$ contains a vertex $v$ such that $N(v)$
  is isomorphic to $K_{2,3}$. Let us label the vertices of $N(v)$
    according to Figure~\ref{f:neighborhoods5}. Now let us focus on $N(w_1)$.
    It contains a subgraph formed by the vertices $v$, $x_1$, $x_2$ and $x_3$
    isomorphic to $K_{1,3}$. Therefore $N(w_1)$ cannot be isomorphic to $C_5$,
     so it must be isomorphic to $K_{2,3}$. Now we focus on $N(x_1)$.
    By checking $N(v)$, we get that $\deg_{N(x_1)}v = 2$. By an analogous
    argument, $\deg_{N(x_1)} w_1 = 2$ since we already know that $N(w_1)$ is
    isomorphic to $K_{2,3}$, and $w_1v$ is an edge in $N(x_1)$. Therefore $N(x_1)$ cannot be isomorphic to
    $K_{2,3}$ which implies that it is isomorphic to $C_5$. Analogously, we
    deduce that $N(x_2)$ and $N(x_3)$ are isomorphic to $C_5$.

As $G$ is a connected $5$-regular graph,
  in order to show
  that $G$ is isomorphic to $C_5 \star I_3$, it is sufficient to show that
  $N(x_1) = N(x_2) = N(x_3)$. We will show $N(x_1) = N(x_2)$ and the other
  equality $N(x_1) = N(x_2)$ will be analogous. Let us again focus on $N(v)$.
  The two edges $vw_1$ and $vw_2$ belong simultaneously to $N(x_1)$ and
  $N(x_2)$. Now, if we refocus to $N(w_1)$, we see that the two edges of
  $N(x_1)$ incident with $w_1$ belong also to $N(x_2)$, as $w_1$ has degree $5$ in $G$.
  By repeating
  this argument for $w_2$, we get that $N(x_1)$ and $N(x_2)$ share a $4$-path on $5$ vertices. As both $N(x_1)$ and
  $N(x_2)$ are isomorphic to $C_5$ we conclude
  $N(x_1) = N(x_2)$.

  Finally, it remains to consider the case that the
  open neighborhood of every vertex of $G$ is isomorphic to $C_5$. In this case, the clique complex
  $\cl(G)$ is a closed triangulated surface without boundary. Let $n$ be the
  number of vertices of $G$. By double-counting, $\cl(G)$ contains $\frac 52 n$
  edges and $\frac 53 n$ triangles. Therefore, the Euler characteristic
  $\chi(\cl(G))$ equals $n - \frac52n + \frac53n = \frac n6$.
  In particular,  $\chi(\cl(G))$ is positive; therefore $\cl(G)$ must be the
  sphere or the projective plane. The case of projective plane cannot occur,
  because in such case, we would have $n = 6\chi(\cl(G)) = 6$, forcing $G=K_6$ in the unique $6$-vertex triangulation of the projective plane; but then $\cl(G)$ is the $5$-simplex, a contradiction.
  Hence we know that $\cl(G)$ is the sphere and $n =
  6\chi(\cl(G)) = 12$. However, it is well known that the only $5$-regular
  graph that triangulates the sphere is the graph of the icosahedron.
\end{proof}

It remains to rule out the two cases from the previous lemma as minimal
counterexamples.

\begin{lemma}
\label{l:not_ico}
  Neither $C_5 \star I_3$ nor $G_{\ico}$ is a counterexample to
  Theorem~\ref{t:K4}.
\end{lemma}

\begin{figure}
\begin{center}
\includegraphics{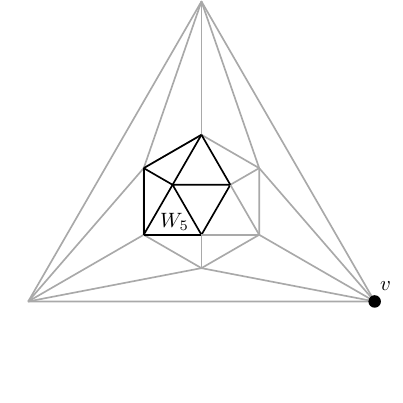}
\caption{A vertex $v$ of the icosahedron and the graph formed by the
non-neighbors of $v$.}
\label{f:icosahedron}
\end{center}
\end{figure}

\begin{proof}
  It is easy to compute that $\bb(C_5 \star I_3) = 2$ since the corresponding
  independence complex is the union of a $5$-cycle and a triangle. Since $2 <
  2^{8/3} = \Theta_2^8$, we get that $C_5 \star I_3$ is not a counterexample to
  Theorem~\ref{t:K4}.

It is a bit harder to determine $\bb(G_{\ico})$ precisely since the
corresponding independence complex is $2$-dimensional. Let $K$ be the
independence complex of $G_{\ico}$. Let $v$ be an arbitrary vertex of $G_{\ico}$.
Then $v$ is not incident to a subgraph of $G_{\ico}$ forming the wheel graph
$W_5$; see Figure~\ref{f:icosahedron}. This gives that the link $\lk_{K}
v$ is the independence complex of $W_5$, that is, the disjoint union $C_5 + I_1$
of $C_5$ and a vertex. Let $K'$ be the complex obtained from $K$ by
removing all edges which are not incident to any triangle. This means removing
$6$ edges. Then the link of every vertex of $K'$ is isomorphic to $C_5$.
Now, the same reasoning as in the last part of the proof of
Lemma~\ref{l:deg5_cases} gives that $K'$ is the boundary of the
icosahedron (but the vertices are significantly permuted when compared to the
icosahedron for $G_{\ico})$. This implies that
$K$ is homotopy equivalent to the wedge
of one $2$-sphere and six $1$-spheres. We obtain $\bb(G_{\ico}) = 7$. Given that
$7 < 8 < 2^{12/3} = \Theta_2^{12}$, we get that $G_{\ico}$ is not a
counterexample to Theorem~\ref{t:K4}.
\end{proof}

Now we conclude everything and obtain the final result.

\begin{proof}[Proof of Theorem~\ref{t:K4}]
  For contradiction, there is a minimal counterexample $G$ to
  Theorem~\ref{t:K4}. By Proposition~\ref{p:5reg}, $G$ is $5$-regular. By
  Lemmas~\ref{l:5edges} and~\ref{l:deg5_cases}, $G$ must be isomorphic to
  $G_{\ico}$ or to $C_5 \star I_3$. However, Lemma~\ref{l:not_ico} excludes
  these two options.
\end{proof}

\end{document}